\documentclass[a4paper,11pt]{amsart}
\usepackage[matrix,arrow,tips,curve]{xy}
\usepackage[english]{babel}
\usepackage{amsmath}
\usepackage{amssymb}
\usepackage{tikz}
\usepackage{enumerate}
\usepackage{graphicx}
\usepackage{amsthm}
\usepackage[colorlinks=true,linkcolor=blue,citecolor=green]{hyperref}
\usetikzlibrary{trees}
\usetikzlibrary{arrows}
\usepackage[pangram]{blindtext}
\newtheorem{claim}{Claim}[section]
\newtheorem{thm}{Theorem}[section]

\newtheorem{Conjecture}[thm]{Conjecture}
\newtheorem*{thm*}{Theorem}
\newtheorem{remark}{Remark}[section]

\newtheorem*{acknowledgment}{Acknowledgment}

\newtheorem{theorem}[thm]{Theorem} 
\newtheorem{lema}[thm]{Lemma}
\newtheorem{question}[thm]{Question}
\newtheorem{cor}[thm]{Corollary}
\newtheorem{corollary}[thm]{Corollary}
\newtheorem{prop}[thm]{Proposition}
\newtheorem{defi}{Definition}[section]

\newtheorem{assumption}{Assumption}[section]

\definecolor{wwwwww}{rgb}{0.4,0.4,0.4}

\DeclareMathOperator{\Ric}{Ric}

\DeclareMathOperator{\dv}{div}

\hypersetup{pdfpagemode=UseNone}
\hypersetup{pdfstartview=FitH}
		
\numberwithin{equation}{section}

\title{Divergence Identity for the scalar curvature and Rigidity of Codazzi Tensors}

\author[Xu CHeng]{Xu Cheng}
\address{Instituto de Matem\'atica e Estat\'istica, Universidade Federal Fluminense, S\~ao Domingos,
Niter\'oi, RJ 24210-201, Brazil}
\email{xucheng@id.uff.br}
\author[Andrés Lipa]{Andrés Lipa}
\address{Instituto de Matem\'atica e Estat\'istica, Universidade Federal Fluminense, S\~ao Domingos,
Niter\'oi, RJ 24210-201, Brazil}
\email{linkin2388@gmail.com}
\author[Detang Zhou]{Detang Zhou}
\address{ Instituto de Matem\'atica e Estat\'istica, Universidade Federal Fluminense, S\~ao Domingos,
Niter\'oi, RJ 24210-201, Brazil}
\email{zhoud@id.uff.br}
\date{}
\thanks{X. Cheng was partially supported by CNPq/Brazil [Grant:  314504/2023-0]. }
\thanks{D. Zhou was partially supported by CNPq/Brazil [Grant:  308067/2023-1] and FAPERJ/ Brazil [Grant: E-26/200.386/2023].}
\begin{document}
\begin{abstract}
We introduce a local vector field on an $n$-dimensional Riemannian manifold, defined as the sum of the covariant derivatives of a local orthonormal frame,  and derive an explicit identity for its divergence, decomposed into a scalar curvature term and an auxiliary term involving connection coefficients. This result is applied to rigidity problems for Codazzi symmetric tensors. In particular, we give a new proof of a Tang-Yan theorem, which states that on a  closed $n$-dimensional manifold with nonnegative scalar curvature, a smooth Codazzi symmetric tensor whose trace invariants up to order $n-1$ are constant must have constant eigenvalues. We also obtain further rigidity results under assumptions on elementary symmetric functions of the eigenvalues, with applications to the isoparametric rigidity of closed hypersurfaces in the unit sphere.

\end{abstract}
\maketitle
\section{Introduction}\label{intro}

Inspired by classical moving-frame calculations, we consider a vector field on an $n$-dimensional Riemannian manifold $M^n$, obtained by summing the covariant derivatives 
\(\nabla_{e_i}e_i\) along a locally defined orthonormal frame $\{e_1,\dots,e_n\}$. Equivalently, this vector field is metrically dual to the trace of the Levi-Civita connection 1-form with respect to the chosen orthonormal frame.   
We derive an explicit expression for its divergence and relate it to the scalar curvature of $M^n$. More precisely, we prove the following result.

\begin{theorem}\label{thm:divX-intro2}
 Let  $\{e_1,\dots,e_n\}$ be  an  orthonormal frame on an open subset  $D$ of an $n$-dimensional Riemannian manifold $M^n$.
Let $X$ be the vector field on $D$ defined by 
\begin{equation}
X:=\sum_{i=1}^n\nabla_{e_i}e_{i}.
\end{equation}
Then the divergence of $X$ satisfies
\begin{equation}\label{eq:divX}
\dv (X) =\frac12 S_M - \Psi ,
\end{equation}
where $S_M$ denotes the scalar curvature of $M^n$, and 
\begin{equation}
\Psi = \sum_{\substack{i<j \\ k\neq i,j}}^n\big(\Gamma_{ii}^k\Gamma_{jj}^k - \Gamma_{ij}^k\Gamma_{ji}^k \big),  \ \Gamma_{ij}^k=\langle \nabla_{e_i}e_{j}, e_k\rangle.
\end{equation}
\end{theorem}

The identity \eqref{eq:divX} is quite general, expressing the divergence of the vector field $X$ as the sum of a scalar curvature term and an auxiliary term involving connection coefficients. It provides a flexible analytic tool: the scalar curvature appears with a fixed coefficient, while $\Psi$ can be estimated or controlled under various geometric conditions. In this paper, we show that such control is precisely what is needed to streamline the proofs of several rigidity results for Codazzi symmetric tensors and to extend these results. At the same time, Theorem \ref{thm:divX-intro2} is of independent interest.

We begin by recalling a rigidity theorem in dimension three due to  Almeida and Brito. In \cite{Almeida-Brito90}, they  proved the following result.
\begin{theorem}[Almeida--Brito \cite{Almeida-Brito90}]\label{A-B90}
Let $M$ be a closed three-dimensional Riemannian manifold. Suppose that 
$\mathfrak{a}$ is a smooth symmetric tensor field of type $(0,2)$ on $M$ and $A$ is its associated
$(1,1)$-tensor field. Assume  that:
\begin{enumerate}
    \item $S_{M} \geq 0$;
    \item $\mathfrak{a}$ is Codazzian;
    \item both $\operatorname{tr}(A)$ and $\operatorname{tr}(A^{2})$ are constant.
\end{enumerate}
Then $\operatorname{tr}(A^{3})$ is  constant, and thus the eigenvalues of $A$
are constant.

\end{theorem}
As a consequence, it was shown in \cite{Almeida-Brito90} that if $M$ is a closed three-dimensional hypersurface in a space form with constant mean curvature and constant scalar curvature $S_M\ge 0$, then $M$ is isoparametric. Later, Chang \cite{Chang93a} and Cheng and Wan \cite{Cheng-Wan} independently strengthened this corollary by proving that the condition $S_M \ge 0$ follows automatically from the assumptions of constant mean curvature and constant scalar curvature.

Subsequently, several works addressed the extension of Theorem~\ref{A-B90} to higher dimensions; see, for instance, 
 Lusala, Scherfner and de Sousa \cite{Lusala-Scherfner-Sousa05}, Scherfner, Vrancken and Weiss \cite{SVW}, Scherfner, Weiss, and Yau \cite{SWY}, Deng,  Gu and  Wei \cite{DGW}, Tang, Wei and Yan \cite{Tang-Wei-Yan20},  and more recently Tang and Yan \cite{Tang-Yan23}.  In \cite{Tang-Yan23}, the following theorem was proved.

\begin{theorem}[Tang--Yan \cite{Tang-Yan23}]\label{T-Y}
Let $M^n$ $(n>3)$ be a closed $n$-dimensional Riemannian manifold.  
Suppose that $\mathfrak{a}$ is a smooth symmetric tensor field of type $(0,2)$ on $M$ and $A$ is its associated
$(1,1)$-tensor field. Assume  that:
\begin{enumerate}
    \item $S_{M} \geq 0$;
    \item $\mathfrak{a}$ is Codazzian;
    \item $\operatorname{tr}(A^k)$ is constant for $k=1,\dots,n-1$.
\end{enumerate}
Then  $\operatorname{tr}(A^{n})$ is constant and thus all eigenvalues of $A$ are constant.  
Moreover, if $A$ has $n$ distinct eigenvalues somewhere on $M$, then $S_M\equiv0$.
\end{theorem}
An application of Theorem \ref{T-Y} to hypersurfaces in the unit sphere yields the following rigidity result.

\begin{corollary}[Tang--Yan \cite{Tang-Yan23}]
Let $M^n$ $(n>3)$ be a closed hypersurface in the unit sphere $S^{n+1}$. 
If the following conditions are satisfied:
\begin{enumerate}
    \item $S_{M} \geq 0$;
    \item $\sum_{i=1}^n \lambda_i^k \ \text{is constant for each } k=1,\dots,n-1$, where $\ \lambda_1,\dots,\lambda_n$ are the principal curvatures of $M^n$, 
\end{enumerate}
then $M^n$ is isoparametric. Moreover, if $M^n$ has $n$ distinct principal curvatures at some point, then $S_M\equiv0$.
\end{corollary}

The  above  results are related to the unsolved Chern conjecture: 
\begin{Conjecture}[Chern]
     Let $M^n$  be a closed, minimally immersed hypersurface of the unit sphere $\mathbb{S}^{n+1}$ with constant scalar curvature. Then $M^n$ is isoparametric.
\end{Conjecture}

For various results related to this conjecture and further background on isoparametric hypersurfaces in spheres, see, for example, \cite{Tang-Yan23}.

In this paper, we give alternative proofs of Theorem \ref{T-Y} and of a special case of Theorem~\ref{T-Y} that was previously obtained by Tang--Wei--Yan \cite{Tang-Wei-Yan20} (see Theorem \ref{T-W-Y}). Our proof has several new features compared with those in \cite{Tang-Yan23} and \cite{Tang-Wei-Yan20}: First, we introduce the vector field $X$ and derive the identity~\eqref{eq:divX} through a detailed computation involving connection coefficients. Then the use of \eqref{eq:divX} allows us to avoid the use of an $(n-1)$-form, which was employed previously in \cite{Almeida-Brito90} for $n=3$ and later in \cite{Tang-Wei-Yan20} and \cite{Tang-Yan23} for 
$n>3$, and lets the proof proceed in a clearer, step-by-step manner.
Second, we make effective use of the derivatives--up to second order--of the characteristic polynomial $P(x)$ of the eigenvalue functions  of the $(1,1)$-tensor $A$. As a result, the introduction of Proposition~\ref{lem:partial-fraction-identity} and the use of Lemma~\ref{lem-Lr-h} simplify the argument relative to that in \cite{Tang-Yan23}.

In this paper, we also study a different setting in which the elementary  symmetric functions $\sigma_k(A)$ of the eigenvalues  of the $(1,1)$-tensor field $A$ are constant for all $k\neq n-1$. 
More precisely, we prove the following results.

\begin{theorem}\label{teo:M^n-(n-1)-intro2}
Let $M^n$ $(n> 3)$ be a closed $n$-dimensional Riemannian manifold. Suppose that $\mathfrak{a}$ is a smooth Codazzi symmetric $(0,2)$-tensor field  on $M$,  and let $A$ be its associated
$(1,1)$-tensor field.  
Assume that:
\begin{enumerate}[(i)]
    \item $S_M \ge 0$;
    \item $\sigma_k(A)$ is constant for $k=1,\dots,n-2$,  $\sigma_n(A)$ vanishes, and $\sigma_{n-1}$ is nonzero everywhere;
    \item all eigenvalues of $A$ are either nonnegative everywhere or nonpositive everywhere.
\end{enumerate}
Then $\sigma_{n-1}(A)$ is constant, and hence all eigenvalues of $A$ are constant.  Moreover, if the eigenvalues of $A$ are distinct somewhere on $M$, then $S_M\equiv0$.
\end{theorem}
More generally, we obtain the following result.
\begin{thm} 
\label{teo:M^n-(n-1)-m}
Let $M^n$ $(n> 3)$ be a closed $n$-dimensional Riemannian manifold. Suppose that $\mathfrak{a}$ is a smooth Codazzi symmetric $(0,2)$-tensor field  on $M$,  and let $A$ be its associated
$(1,1)$-tensor field.  
Assume that:
\begin{enumerate}[(i)]
    \item $S_M \ge 0$;
    \item $\sigma_k(A)$ is constant for $k=1,\dots,n-2$;
    \item the maximal eigenvalue of $A$   or the minimal eigenvalue of $A$ is simple and constant.
    
\end{enumerate}
Then all eigenvalues of $A$ are constant.  Moreover, if the eigenvalues of $A$ are distinct somewhere, then $S_M\equiv0$.
\end{thm}

We apply Theorem \ref{teo:M^n-(n-1)-m} to obtain the following rigidity result for closed hypersurfaces in $\mathbb S^{n+1}$. 
\begin{corollary}
\label{cor:hypersphere}
Let $M^n$ $(n>3)$ be a closed hypersurface in the $(n+1)$-dimensional unit sphere $\mathbb S^{n+1}$.   Assume that:
\begin{enumerate}[(i)]
  \item $S_M\ge0$ on $M$;
  \item the maximal principal curvature   or the minimal principal curvature is simple and constant;
  \item each $r$-th  mean curvatures $H_r$ is constant for $r=1,\dots,n-2$.
  
\end{enumerate}
Then  $M$ is isoparametric.
Moreover, if the principal curvatures are distinct somewhere, then $S_M\equiv0$. 
\end{corollary}

In this paper, we also present applications to other Codazzi symmetric $(0,2)$-tensors.   Recall that a Riemannian manifold $M^n$ is said to have \emph{harmonic curvature} if the Riemann  tensor is  Codazzian.
A corollary of Theorem \ref{teo:M^n-(n-1)-intro2} is the following rigidity property for  Ricci tensor on 
the Riemannian manifolds with harmonic curvature.
\begin{corollary}
\label{cor:ricci}
Let $(M^n,g)$, $n>3$, be a closed Riemannian manifold with harmonic curvature and scalar curvature $S\ge 0$.  
Let $\operatorname{Ric}$ denote the Ricci tensor.
Assume that:
\begin{enumerate}[(i)]
  \item the elementary symmetric functions $\sigma_k(\Ric)$ are constant for $k=1,\dots,n-2$,  $\sigma_n(\Ric)\equiv 0$, and  $\sigma_{n-1}(\Ric)$ is nonzero everywhere;
  \item the eigenvalues of the Ricci tensor are everywhere nonnegative.
\end{enumerate}
Then all eigenvalues of the Ricci tensor are constant. Moreover, if the Ricci tensor has $n$ distinct eigenvalues somewhere, then $S\equiv 0$.
\end{corollary}

\begin{remark}\label{rema-1}
   It is worth pointing out that, in the study of  gradient shrinking Ricci solitons,  H.-D. Cao proposed a rigidity conjecture asserting that a complete \(n\)-dimensional gradient shrinking Ricci soliton \((M^n,g,f)\) with constant scalar curvature must be rigid, that is, a finite quotient of \(N^k\times \mathbb{R}^{n-k}\) for some Einstein manifold \(N\) of positive scalar curvature. An observation is that a constant scalar curvature Ricci shrinker with nonconstant potential function necessarily  satisfies  that $\sigma_1, \sigma_2$ are constant and  $\sigma_n=0$.  Without assuming the harmonicity of the curvature tensor,  Cheng and Zhou \cite{Cheng-Zhou23}, together with the work of Petersen and Wylie \cite{PW09} and Fernández-López and García-Río \cite{FG16},  confirmed this conjecture in dimension four, proving that a four-dimensional complete noncompact gradient shrinking Ricci soliton has constant scalar curvature if and only if it is rigid, i.e., it is isometric to either an Einstein manifold, or to a finite quotient of the Gaussian shrinking soliton \(\mathbb{R}^4\), \(\mathbb{S}^2\times \mathbb{R}^2\), or \(\mathbb{S}^3\times \mathbb{R}\). 
\end{remark}

Remark \ref{rema-1}  suggests that the constancy of several symmetric functions of a tensor field may strongly restrict its eigenvalue structure, thereby motivating the following question.

\begin{question}
Suppose $\mathfrak{a}$ is a smooth symmetric tensor field of type \((0,2)\) on \(M^n\), and let \(A\) be the associated \((1,1)\)-tensor field.
Under what conditions, and for which choice of \(n-1\) of the elementary symmetric functions \(\sigma_k(A)\), does the constancy of these functions imply that all eigenvalues of \(A\) are constant?
\end{question}

This paper is organized as follows.
Section \ref{intro} is the introduction.
Section \ref{V} is devoted to the proof of Theorem \ref{thm:divX-intro2}.
In Section \ref{Apli-Tensor}, some preliminaries on $(0,2)$-tensors are presented.
Section~\ref{lambda-ij} establishes several auxiliary results, some of which are used to prove Theorem~\ref{teo:M^n1-1}, while others are used, together with Theorem~\ref{teo:M^n1-1}, to prove Theorem~\ref{T-Y}. In Section~\ref{sec:5}, the proofs of Theorem~\ref{teo:M^n1-1} and Theorem~\ref{T-Y} are given.
In Section \ref{sec:6}, we prove Theorems  \ref{teo:M^n-(n-1)-intro2} and \ref{teo:M^n-(n-1)-m}.
In Section \ref{sec:app}, the applications are discussed and Section \ref{appendix} is the appendix.

\begin{acknowledgment}
The authors thank Professor F.~Fontenele for helpful comments and suggestions.
\end{acknowledgment}

\section{A Vector field}\label{V}

Let $M^n$ be an $n$-dimensional Riemannian manifold. We define the vector field $X$ introduced in the Introduction and prove Theorem \ref{thm:divX-intro2}.

\begin{theorem}[Theorem \ref{thm:divX-intro2}]
\label{prop:X,div(X)}
    Let  $\{e_1,\dots,e_n\}$ be  an  orthonormal frame  on an open subset  $D$  of an $n$-dimensional  Riemannian manifold $M^n$.
Let $X$ be the vector field defined by
\begin{equation}
\label{def:X}
X:=\sum_{i=1}^n\nabla_{e_i}e_{i}.
\end{equation}
Then the divergence of $X$ satisfies
\begin{equation}
\label{div(X)}
\dv (X) =\frac12 S_M - \Psi ,
\end{equation}
where $S_M$ denotes the scalar curvature of $M$, and
\begin{equation}
\label{Psi}
\Psi := \sum_{\substack{i<j \\ k\neq i, j}}^n \big( \Gamma_{ii}^k\Gamma_{jj}^k - \Gamma_{ij}^k\Gamma_{ji}^k \big),    \ \Gamma_{ij}^k=\langle \nabla_{e_i}e_{j}, e_k\rangle.
\end{equation}
\end{theorem}  

\begin{proof} Since
\begin{equation}
\label{eq:X}
X=\sum_{i=1}^n\big(\sum_{j=1}^n\langle \nabla_{e_i}e_i, e_j\rangle e_j\big) =\sum_{j=1}^n\big(\sum_{i=1}^n\langle \nabla_{e_i}e_i, e_j\rangle \big)e_j,
\end{equation}
 the divergence of $X$ is 
\begin{eqnarray*}
\dv(X)&=& \sum_{j=1}^n\langle \nabla_{e_j}X, e_j\rangle=\sum_{j=1}^n(e_j\langle X, e_j\rangle-\langle X, \nabla_{e_j} e_j\rangle)\\
&=&\sum_{j=1}^n\big[e_j\big(\sum_{i=1}^n\langle \nabla_{e_i}e_i, e_j\rangle\big)-\langle \sum_{k=1}^n\big(\sum_{i=1}^n\langle \nabla_{e_i}e_i, e_k\rangle\big)e_k, \nabla_{e_j}e_j\rangle\big]\\
&=&\sum_{i,j=1}^ne_j\langle \nabla_{e_i}e_i, e_j\rangle- \sum_{i,j,k=1}^n \Gamma_{ii}^k\Gamma_{jj}^k.
\end{eqnarray*}
Since  $\langle e_{j}, e_k\rangle=\delta_{jk}$ and $\Gamma_{ij}^k=\langle \nabla_{e_i}e_j,e_k\rangle$, we have
$\Gamma_{ii}^{i}=\langle \nabla_{e_i}e_i, e_i\rangle=0$. Then
\begin{eqnarray}
\label{eq:div(X)}
\dv(X)
&=&\sum_{ i< j}^n e_j\langle \nabla_{e_i}e_i, e_j\rangle+\sum_{i>j}^n e_j\langle \nabla_{e_i}e_i, e_j\rangle- \sum_{ k\neq i,j}^n \Gamma_{ii}^k\Gamma_{jj}^k \nonumber\\
&=& \sum_{ i<j}^n(  e_j\langle \nabla_{e_i}e_i, e_j\rangle+e_i\langle \nabla_{e_j}e_j, e_i\rangle) - \sum_{\substack{i<j \\ k\neq i,j}}^n \Gamma_{ii}^k\Gamma_{jj}^k-\sum_{\substack{i>j \\ k\neq i,j}}^n\Gamma_{ii}^k\Gamma_{jj}^k-\sum_{k\neq i}^n(\Gamma_{ii}^k)^2\nonumber\\
&=&\sum_{ i<j}^n(  e_j\langle \nabla_{e_i}e_i, e_j\rangle+e_i\langle \nabla_{e_j}e_j, e_i\rangle) - 2\sum_{\substack{i<j \\ k\neq i,j}}^n \Gamma_{ii}^k\Gamma_{jj}^k-\sum_{k\neq i}^n(\Gamma_{ii}^k)^2.
\end{eqnarray}
Fix $i$ and $j$. We next compute the two derivative terms appearing above.
\begin{eqnarray}
\label{Gamma_jj^i;i}
 e_i\langle \nabla_{e_j}e_j, e_i\rangle 
&=& \langle \nabla_{e_i} \nabla_{e_j}e_j, e_i\rangle +\langle \nabla_{e_i}e_i, \nabla_{e_j}e_j\rangle  \nonumber\\
&=& \langle \nabla_{e_i} \nabla_{e_j}e_j, e_i\rangle +\sum_{k=1}^n \Gamma_{ii}^k \Gamma_{jj}^k.
\end{eqnarray}
\begin{eqnarray}
\label{eq:Gamma_ii^j;j}
e_j\langle \nabla_{e_i}e_i, e_j\rangle
&=& -e_j\langle \nabla_{e_i}e_j, e_i\rangle =-\langle \nabla_{e_j} \nabla_{e_i}e_j, e_i\rangle -\langle \nabla_{e_i}e_j, \nabla_{e_j}e_i\rangle \nonumber\\
&=& -\langle \nabla_{e_j} \nabla_{e_i}e_j, e_i\rangle -\sum_{k=1}^n\Gamma_{ij}^k \Gamma_{ji}^k.
\end{eqnarray}
It follows from (\ref{Gamma_jj^i;i}) and (\ref{eq:Gamma_ii^j;j}) that 
\begin{eqnarray}\label{eq:curvature}
    e_j\langle \nabla_{e_i}e_i, e_j\rangle+ e_i\langle \nabla_{e_j}e_j, e_i\rangle&=& \langle \nabla_{e_i} \nabla_{e_j}e_j -\nabla_{e_j} \nabla_{e_i}e_j, e_i\rangle -\sum_{k=1}^n \Gamma_{ij}^k \Gamma_{ji}^k +\sum_{k=1}^n \Gamma_{ii}^k \Gamma_{jj}^k\nonumber\\
    &=&\langle R(e_i,e_j)e_j,e_i \rangle +\langle \nabla_{[e_i,e_j]}e_j,e_i\rangle -\sum_{k=1}^n \Gamma_{ij}^k \Gamma_{ji}^k +\sum_{k=1}^n\Gamma_{ii}^k \Gamma_{jj}^k\nonumber\\
    &=&\langle R(e_i,e_j)e_j,e_i \rangle +\sum_{k=1}^n (\Gamma_{ij}^k - \Gamma_{ji}^k)\Gamma_{kj}^i - \sum_{k=1}^n \Gamma_{ij}^k \Gamma_{ji}^k +\sum_{k=1}^n \Gamma_{ii}^k \Gamma_{jj}^k.
\end{eqnarray}
Substituting \eqref{eq:curvature} into (\ref{eq:div(X)}) yields
\begin{eqnarray}
\label{div}
\dv(X) 
&=& \frac{1}{2}S_M + \sum_{i<j}^n\big[ \sum_{k=1}^n (\Gamma_{ij}^k - \Gamma_{ji}^k)\Gamma_{kj}^i - \sum_{k=1}^{n} \Gamma_{ij}^k \Gamma_{ji}^k +\sum_{k=1}^{n} \Gamma_{ii}^k \Gamma_{jj}^k \big]  - 2\sum_{\substack{i<j \\ k\neq i,j}}^n \Gamma_{ii}^k\Gamma_{jj}^k-\sum_{k\neq i}^n(\Gamma_{ii}^k)^2\nonumber\\
&=& \frac{1}{2}S_M + \sum_{\substack{i<j \\ k\neq i,j}}^n [(\Gamma_{ij}^k - \Gamma_{ji}^k)\Gamma_{kj}^i -\Gamma_{ij}^k \Gamma_{ji}^k -\Gamma_{ii}^k \Gamma_{jj}^k ] +\sum_{i<j}^n[(\Gamma_{ij}^i)\Gamma_{ij}^i +(-\Gamma_{ji}^j)\Gamma_{jj}^i ]-\sum_{i\neq j}^n (\Gamma_{ii}^j)^2\nonumber\\
&=& \frac{1}{2}S_M + \sum_{\substack{i<j \\ k\neq i,j}}^n [(\Gamma_{ij}^k - \Gamma_{ji}^k)\Gamma_{kj}^i -\Gamma_{ij}^k \Gamma_{ji}^k -\Gamma_{ii}^k \Gamma_{jj}^k ] +\sum_{i<j}^n[(\Gamma_{ii}^j)^2 + (\Gamma_{jj}^i)^2]-\sum_{i\neq j}^n (\Gamma_{ii}^j)^2
\end{eqnarray}
In the second equality of \eqref{div}, we separated the terms with $k\neq i,j$ from those with $k=i,j$, and used $\Gamma_{lk}^k=0$.   In the last line of \eqref{div}, we used $-\Gamma_{ji}^j=\Gamma_{jj}^i$.

\begin{claim}
\label{claim:gamma} It holds that
    \[
    \sum_{\substack{i<j \\ k\neq i,j}}^n  (\Gamma_{ij}^k-\Gamma_{ji}^k)\Gamma_{kj}^i = 2 \sum_{\substack{i<j \\ k\neq i,j}}^n\Gamma_{ij}^k\Gamma_{ji}^k.
    \]
\end{claim}
Indeed,
\begin{eqnarray}\label{eq:eq:Gamma_ijkGamma_jik-r}
\sum_{\substack{i<j \\ k \neq i,j}}^n\Gamma_{ij}^k \Gamma_{ji}^k &=& \sum_{i<j<k}^n \Gamma_{ij}^k \Gamma_{ji}^k + \sum_{i<k<j}^n \Gamma_{ij}^k \Gamma_{ji}^k + \sum_{k<i<j}^n \Gamma_{ij}^k \Gamma_{ji}^k \nonumber\\
&=& \sum_{i<j<k}^n \Gamma_{ij}^k \Gamma_{ji}^k + \sum_{i<j<k}^n \Gamma_{ik}^j \Gamma_{ki}^j + \sum_{i<j<k}^n \Gamma_{jk}^i \Gamma_{kj}^i \nonumber\\
&=& \sum_{i<j<k}^n ( \Gamma_{ij}^k \Gamma_{ji}^k +\Gamma_{ik}^j \Gamma_{ki}^j +\Gamma_{jk}^i \Gamma_{kj}^i).
\end{eqnarray}
On the other hand, using $\Gamma_{ij}^k\Gamma_{kj}^i=\Gamma_{ik}^j \Gamma_{ki}^j$ and $-\Gamma_{ji}^k\Gamma_{kj}^i=\Gamma_{jk}^i \Gamma_{kj}^i$, we have
\begin{eqnarray}
\label{eq:Gamma_ijkGamma_jik}
\sum_{\substack{i<j \\ k\neq i,j}}^n (\Gamma_{ij}^k-\Gamma_{ji}^k)\Gamma_{kj}^i &=& \sum_{\substack{i<j \\ k\neq i,j}}^n (\Gamma_{ik}^j \Gamma_{ki}^j +\Gamma_{jk}^i \Gamma_{kj}^i) \nonumber\\
&=& \sum_{i<j<k}^n (\Gamma_{ik}^j \Gamma_{ki}^j +\Gamma_{jk}^i \Gamma_{kj}^i) + \sum_{i<k<j}^n(\Gamma_{ik}^j \Gamma_{ki}^j +\Gamma_{jk}^i \Gamma_{kj}^i) + \sum_{k<i<j}^n (\Gamma_{ik}^j \Gamma_{ki}^j +\Gamma_{jk}^i \Gamma_{kj}^i) \nonumber\\
&=& \sum_{i<j<k}^n (\Gamma_{ik}^j \Gamma_{ki}^j +\Gamma_{jk}^i \Gamma_{kj}^i) + \sum_{i<j<k}^n (\Gamma_{ij}^k \Gamma_{ji}^k +\Gamma_{kj}^i \Gamma_{jk}^i) + \sum_{i<j<k}^n (\Gamma_{ji}^k \Gamma_{ij}^k +\Gamma_{ki}^j \Gamma_{ik}^j) \nonumber\\
&=& 2\sum_{i<j<k}^n (\Gamma_{ij}^k \Gamma_{ji}^k +\Gamma_{ik}^j \Gamma_{ki}^j +\Gamma_{jk}^i \Gamma_{kj}^i ).
\end{eqnarray}
Thus Claim \ref{claim:gamma} holds from \eqref{eq:eq:Gamma_ijkGamma_jik-r} and \eqref{eq:Gamma_ijkGamma_jik}.

By Claim \ref{claim:gamma} and (\ref{div}), we have
\begin{eqnarray*}
\dv(X) &=& \frac{1}{2}S_M + \sum_{\substack{i<j \\ k\neq i,j}}^n ( \Gamma_{ij}^k \Gamma_{ji}^k -\Gamma_{ii}^k \Gamma_{jj}^k ) +\sum_{i<j}^n[(\Gamma_{ii}^j)^2 + (\Gamma_{jj}^i)^2]-\sum_{i\neq j}^n (\Gamma_{ii}^j)^2\\
&=& \frac{1}{2}S_M - \sum_{\substack{i<j \\ k\neq i,j}}^n(\Gamma_{ii}^k \Gamma_{jj}^k -\Gamma_{ij}^k \Gamma_{ji}^k ) +\sum_{i\neq j}^n (\Gamma_{ii}^j)^2-\sum_{i\neq j}^n (\Gamma_{ii}^j)^2=\frac{1}{2}S_M - \Psi,
\end{eqnarray*}
where
\begin{equation*}
\Psi = \sum_{\substack{i<j \\ k\neq i, j}}^n (\Gamma_{ii}^k\Gamma_{jj}^k - \Gamma_{ij}^k\Gamma_{ji}^k )=\sum_{\substack{i<j \\ k\neq i,j}}^n\det
\begin{pmatrix}
\Gamma_{ii}^k & \Gamma_{ij}^k\\
\Gamma_{ji}^k & \Gamma_{jj}^k
\end{pmatrix}.
\end{equation*}
\end{proof}

\begin{remark}
    When $n=2$,  one has $\Psi=0$,  and hence $\dv(X)=K$,   where $K$ is the Gaussian curvature.
\end{remark}
\begin{cor}
\label{cor:X,div(X)}
     Let $D$ be a bounded domain of an $n$-dimensional Riemannian manifold $M^n$ with an orthonormal frame $\{e_1,\hdots, e_n\}$, and let  $X:=\sum_{i=1}^n\nabla_{e_i}e_{i}.$ Then
\begin{equation}
\label{div(X1)}
\int_{\partial D} \langle X,\nu\rangle =\int_D \frac12 S_M -\int_D \Psi ,
\end{equation}
where $S_M$ denotes the scalar curvature of $M$. If $\int_{\partial D} \langle X,\nu\rangle =0$ (in particular, if $\partial D=\emptyset$),   then
\[ \int_D \frac12 S_M =\int_D \Psi. \]
\end{cor}
\begin{remark}
    For any  Riemannian metric on an $n$-dimensional torus, we can always construct a globally defined orthonormal frame.  
\end{remark}

\section{Identities for symmetric 2-tensors and connection coefficients}\label{Apli-Tensor}

As a preliminary, we present in this section several identities relating the covariant derivative of a smooth symmetric $(0,2)$-tensor to the connection coefficients. The results presented here appear, either directly or indirectly, in \cite{Almeida-Brito90} and \cite{Tang-Wei-Yan20}.

\begin{assumption}\label{assumption-3}
     Let $ D$  be an open subset of an $n$-dimensional  Riemannian manifold $M^n$. Let $\{e_1,\ldots,e_n\}$ be an orthonormal frame field on $D$ and $\{\theta_1,\ldots,\theta_n\}$ its dual coframe. Let $\mathfrak{a}$ be a smooth symmetric $(0,2)$-tensor field on $D$, denoted by
\begin{equation*}
\mathfrak{a} = \sum_{i,j=1}^{n} a_{ij} \theta_i \otimes \theta_j,
\end{equation*}
where $a_{ij} = \mathfrak{a}(e_i,e_j)$ are smooth and $a_{ij}=a_{ji}$. 

The covariant derivative of $\mathfrak{a}$ can be written as
\begin{equation*}
\nabla \mathfrak{a} = \sum_{i,j,k=1}^{n} a_{ijk} \theta_i \otimes \theta_j \otimes \theta_k, \ \text{ where } \ a_{ijk} = \nabla \mathfrak{a}(e_i,e_j,e_k)=(\nabla_{e_k}\mathfrak{a})(e_i,e_j).
\end{equation*}

   We further assume that  $\mathfrak{a}$ is diagonal on $D$ with respect to the orthonormal frame  $\{e_1,\hdots,e_n\}$, that is,  $a_{ij} = \lambda_i\delta_{ij}$ on $D$.

It follows that each $\lambda_i=a_{ii}$ is  smooth on $D$ and hence its differential $d\lambda_i$ is a  smooth 1-form, which can be expressed as
\begin{equation*}
d\lambda_i = \sum_{j=1}^{n}\lambda_{ij}\theta_j,
\end{equation*}
 where $\lambda_{ik}$ are smooth functions on $D$. 
 \end{assumption}
 
\begin{lema}
\label{lem:a_ijk}  Under Assumption \ref{assumption-3}, the following statements hold for  $(0,2)$-tensor field $\mathfrak{a}$ on $D$:
\begin{enumerate}[(i)]
    \item If $i \neq j$, then $a_{ijk} = (\lambda_{j} - \lambda_{i})\Gamma_{kj}^i$. In particular $a_{ijj} = (\lambda_{j} - \lambda_{i})\Gamma_{jj}^i$.
    \item If $\lambda_i = \lambda_j$ at $p\in D$, then $\displaystyle\sum_{k=1}^{n}a_{ijk}\theta_k = d(a_{ij})$.
    \item  $a_{iik} = \lambda_{ik}$.
\end{enumerate}
\end{lema}
\begin{proof}
For fixed $i,j$,  we have
\begin{eqnarray*}
a_{ijk} &=& e_k(\mathfrak{a}(e_i,e_j)) - \mathfrak{a}(\nabla_{e_k}e_i,e_j) - \mathfrak{a}(e_i,\nabla_{e_k}e_j) \nonumber\\
&=& e_k(a_{ij}) - \mathfrak{a}\big(\sum_{l=1}^{n}\Gamma_{ki}^l e_l,e_j\big) - \mathfrak{a}\big(e_i,\sum_{l=1}^{n}\Gamma_{kj}^l e_l\big) \nonumber\\
&=& e_k(a_{ij}) - \sum_{l=1}^{n}\Gamma_{ki}^l a_{lj} - \sum_{l=1}^{n}\Gamma_{kj}^l a_{il} \nonumber\\
& = & e_k(a_{ij}) - \Gamma_{ki}^j a_{jj} - \Gamma_{kj}^i a_{ii} \nonumber\\
&=& e_k(a_{ij}) + (\lambda_{j} - \lambda_{i})\Gamma_{kj}^i.
\end{eqnarray*}
Thus
\begin{equation}
\label{eq:a_ijk}
\sum_{k=1}^{n}a_{ijk}\theta_k = d(a_{ij}) + (\lambda_{j} - \lambda_{i})\sum_{k=1}^{n}\Gamma_{kj}^i \theta_k.    
\end{equation}
If $i \neq j$, then   $a_{ij}=0$. From (\ref{eq:a_ijk}) we get
\begin{equation*}
\label{eq:a_ijk_distinct}
a_{ijk} = (\lambda_{j} - \lambda_{i})\Gamma_{kj}^i. 
\end{equation*}
If $\lambda_i = \lambda_j$ at $p$, it follows from (\ref{eq:a_ijk}) that
\begin{equation*}
\sum_{k=1}^{n}a_{ijk}\theta_k = d(a_{ij}).
\end{equation*}
To prove $(iii)$, take   $i=j$.   Then $\lambda_i = \lambda_j$ and from $(ii)$ we have
\begin{equation*}
\sum_{k=1}^{n}a_{iik}\theta_k = d\lambda_i = \sum_{k=1}^{n}\lambda_{ik}\theta_k.
\end{equation*}
Thus, $a_{iik} = \lambda_{ik}$.
\end{proof}
\begin{defi}
   A symmetric $(0,2)$-tensor $\mathfrak{a}$ on a subset $D\subset M$ is called Codazzian  if \[(\nabla_X \mathfrak{a})(Y,Z)=(\nabla_Y \mathfrak{a})(X,Z)\] for all vector fields $X,Y $, and $Z$ on $D$.
    Equivalently, with respect to any local frame $\{e_i\}$, the components  
 $a_{ijk}$  are symmetric in  all indices $i, j, k$.
   
\end{defi}

\begin{lema}
\label{lem:a_is_Codazzian}
Let $\mathfrak{a}$ be $(0,2)$-tensor field on $D$ satisfying Assumption \ref{assumption-3}. Assume that  $\mathfrak{a}$ is Codazzian on $D$ and its eigenvalues $\lambda_i$  are pointwise distinct. Then
     \begin{equation}
    \label{eq:Psi(lambda_ik)}
    \Gamma_{ii}^k=\frac{\lambda_{ik}}{\lambda_{i} - \lambda_{k}},  \text{ for } i\neq k.
    \end{equation}
    and 
    \begin{equation}\label{eq:Psi(p)}
     \sum_{\substack{i<j \\ k\neq i,j}}^n \Gamma_{ij}^k \Gamma_{ji}^k = 0,
     \end{equation}
    
    \end{lema}

\begin{proof}
Since $\mathfrak{a}$ is Codazzian, it holds that
\begin{equation}
\label{eq:a_ijk_symmetric}
a_{ijk}=a_{ikj}=a_{kij}=a_{kji}=a_{jki}=a_{jik}.
\end{equation}
By Lemma \ref{lem:a_ijk} $(i)$ and $(iii)$,  we have that for $i\neq k$,
\begin{eqnarray}
\label{eq:Gamma_ii^kGamma_jj^k}
\Gamma_{ii}^k  &=& \frac{a_{kii}}{\lambda_{i} - \lambda_{k}}=\frac{a_{iik}}{\lambda_{i} - \lambda_{k}} = \frac{\lambda_{ik}}{\lambda_{i} - \lambda_{k}}.
\end{eqnarray}

For fixed $i,j,k$ with $\lambda_i,\lambda_j,\lambda_k$ mutually distinct, from Lemma \ref{lem:a_ijk}-$(i)$ and (\ref{eq:a_ijk_symmetric}) it follows that
\begin{equation}
\Gamma_{ij}^k \Gamma_{ji}^k = \left(\frac{a_{kji}}{\lambda_{j} - \lambda_{k}}\right) \left(\frac{a_{kij}}{\lambda_{i} - \lambda_{k}} \right) = \frac{a_{ijk}^2}{(\lambda_{i} - \lambda_{k})(\lambda_{j} -\lambda_{k})}.
\end{equation}
Then, using that $\lambda_i$'s are distinct, we have 
\begin{eqnarray}
\label{eq:sum_Gamma_ij^kGamma_ji^k}
\sum_{\substack{i<j \\ k \neq i,j}}^n \Gamma_{ij}^k \Gamma_{ji}^k &=& \sum_{\substack{i<j \\ k \neq i,j}}^{n} \frac{a_{ijk}^2}{(\lambda_{i} - \lambda_{k})(\lambda_{j} - \lambda_{k})} \nonumber\\
&=& \sum_{i<j<k}^{n} \frac{a_{ijk}^2}{(\lambda_{i} - \lambda_{k})(\lambda_{j} - \lambda_{k})} + \sum_{i<k<j}^{n} \frac{a_{ijk}^2}{(\lambda_{i} - \lambda_{k})(\lambda_{j} - \lambda_{k})} + \sum_{k<i<j}^{n} \frac{a_{ijk}^2}{(\lambda_{i} - \lambda_{k})(\lambda_{j} - \lambda_{k})} \nonumber\\
&=& \sum_{i<j<k}^{n} \frac{a_{ijk}^2}{(\lambda_{i} - \lambda_{k})(\lambda_{j} - \lambda_{k})} + \sum_{i<j<k}^{n} \frac{a_{ikj}^2}{(\lambda_{i} - \lambda_{j})(\lambda_{k} - \lambda_{j})} + \sum_{i<j<k}^{n} \frac{a_{jki}^2}{(\lambda_{j} - \lambda_{i})(\lambda_{k} - \lambda_{i})} \nonumber\\
&=& \sum_{i<j<k}^{n} a_{ijk}^2 \Big( \frac{1}{(\lambda_{i} - \lambda_{k})(\lambda_{j} - \lambda_{k})} - \frac{1}{(\lambda_{i} - \lambda_{j})(\lambda_{j} - \lambda_{k})} + \frac{1}{(\lambda_{i} - \lambda_{j})(\lambda_{i} - \lambda_{k})} \Big) \nonumber\\
&=& 0.
\end{eqnarray}
The fourth equality in \eqref{eq:sum_Gamma_ij^kGamma_ji^k} followed from (\ref{eq:a_ijk_symmetric}) and the last equality used
\[\frac{1}{(\lambda_{i} - \lambda_{k})(\lambda_{j} - \lambda_{k})} - \frac{1}{(\lambda_{i} - \lambda_{j})(\lambda_{j} - \lambda_{k})} + \frac{1}{(\lambda_{i} - \lambda_{j})(\lambda_{i} - \lambda_{k})}=0.\] 
\end{proof}

\section{Expressions for $\Psi$  and $\langle X, \nabla \sigma_n\rangle$ }  \label{lambda-ij}

We first prove an expression for the derivative  of $n$ smooth distinct real functions on an open subset of $M$.
\begin{prop}\label{lem:lambda_ij-sigma} 
Let $D$ be an open subset of an $n$-dimensional Riemannian manifold $M^n$ with an orthonormal frame $\{e_i\}$ and its dual coframe  $\{\theta_i\}$.  Let $\lambda_1,\dots,\lambda_n$ be smooth pointwise distinct real-valued functions on $D$.  
Let $\sigma_k=\sigma_k(\lambda_1,\dots,\lambda_n)$ denote the $k$-th elementary symmetric function of $\lambda_1, \ldots, \lambda_n$, so that

\[
P(x):=\prod_{i=1}^n (x-\lambda_i)
= x^n - \sigma_1 x^{\,n-1} + \sigma_2 x^{\,n-2} - \cdots + (-1)^n \sigma_n.
\]
If $\sigma_k$ is constant for all $k\neq n$, 
then 
\[
\lambda_{ij} \;=\; (-1)^{\,n+1}\,\frac{(\sigma_n)_j}{P'(\lambda_i)}, 
\]
where
\[ 
d\lambda_i=\sum_{j=1}^n \lambda_{ij}\,\theta_j, \
d\sigma_n=\sum_{j=1}^n (\sigma_n)_j\,\theta_j
 \text{ and }
P'(\lambda_i)=\prod_{k\neq i}(\lambda_i-\lambda_k).
\]
\end{prop}

\begin{proof}
 Using the expansion of $P(x)$ in the elementary symmetric functions, differentiate the identity $P(\lambda_i)=0$ and obtain that
\[
P'(\lambda_i)\,d\lambda_i
+\sum_{k=1}^n \frac{\partial P}{\partial \sigma_k}(\lambda_i)\,d\sigma_k
=0.
\]
Since $\dfrac{\partial P}{\partial \sigma_k}(x)=(-1)^k x^{\,n-k}$, we have
\[
P'(\lambda_i)\,d\lambda_i + \sum_{k=1}^n (-1)^k \lambda_i^{\,n-k}\,d\sigma_k = 0.
\]
Since $\sigma_k$ is constant for all $k\neq n$,  it follows that $d\sigma_k=0$ for these $k$, and  hence 
\[
P'(\lambda_i)\,d\lambda_i + (-1)^n d\sigma_n = 0.
\]
Thus
\[
d\lambda_i = -(-1)^n \frac{d\sigma_n}{P'(\lambda_i)} = (-1)^{\,n+1}\frac{d\sigma_n}{P'(\lambda_i)}.
\]
Writing $d\sigma_n=\sum_{j} (\sigma_n)_j\,\theta_j$ yields the component formula
\[
\lambda_{ij} = (-1)^{\,n+1}\frac{(\sigma_n)_j}{P'(\lambda_i)}.
\]
\end{proof}
 We establish the following polynomial equality for use in the proofs of  Lemmas \ref{lem-Lr-h} and \ref{x-n-1}.
\begin{prop}\label{lem:partial-fraction-identity}
Let $\lambda_1,\ldots,\lambda_m$ $(m\geq2)$ be distinct real numbers and define a polynomial of degree $m$ by
\[
P(x)=\prod_{k=1}^m (x-\lambda_k).
\] 
Then for any $i\in\{1,\ldots,m\}$,
\begin{equation}\label{eq:pf-identity}
\sum_{\substack{k=1 \\ k\neq i}}^m 
\frac{1}{P'(\lambda_k)(\lambda_i-\lambda_k)}
=
-\frac12\cdot\frac{P''(\lambda_i)}{P'(\lambda_i)^2}.
\end{equation}

\end{prop}

\begin{proof}
Since the roots $\lambda_1,\ldots,\lambda_m$ are simple, the rational function
$P(x)^{-1}$ admits the partial fraction decomposition
\[
\frac{1}{P(x)}=
\sum_{k=1}^m \frac{1}{\prod_{i\neq k}^m (\lambda_k-\lambda_i)}\cdot\frac{1}{x-\lambda_k}
=
\sum_{k=1}^m \frac{1}{P'(\lambda_k)}\cdot\frac{1}{x-\lambda_k}.
\]
Separating the term $k=i$, we can write
\[
  \frac{1}{P(x)}
=
\frac{1}{P'(\lambda_i)}\,\frac{1}{x-\lambda_i}
+
\sum_{\substack{k=1 \\ k\neq i}}^m
\frac{1}{P'(\lambda_k)}\,\frac{1}{x-\lambda_k}.  
\]
Then
\begin{equation}\label{eq:frac}
\sum_{\substack{k=1 \\ k\neq i}}^m
\frac{1}{P'(\lambda_k)}\,\frac{1}{x-\lambda_k}
=
\frac{1}{P(x)}-\frac{1}{P'(\lambda_i)}\,\frac{1}{x-\lambda_i}.
\end{equation}
We let $x\to\lambda_i$ on both sides of \eqref{eq:frac}.
The left-hand side converges to
\[
\sum_{\substack{k=1 \\ k\neq i}}^m 
\frac{1}{P'(\lambda_k)(\lambda_i-\lambda_k)}.
\]
To evaluate the right-hand side, we expand $P$ at $\lambda_i$:
\[
P(x)
=
P'(\lambda_i)(x-\lambda_i)
+\frac12 P''(\lambda_i)(x-\lambda_i)^2
+O\!\left((x-\lambda_i)^3\right).
\]
Hence
\[
\frac{1}{P(x)}
=
\frac{1}{P'(\lambda_i)}\frac{1}{x-\lambda_i}
-
\frac12\frac{P''(\lambda_i)}{P'(\lambda_i)^2}
+O(x-\lambda_i).
\]
Subtracting the singular term $\frac{1}{P'(\lambda_i)}(x-\lambda_i)^{-1}$
and letting $x\to\lambda_i$, we obtain
\[
\lim_{x\to\lambda_i}
\left(
\frac{1}{P(x)}-\frac{1}{P'(\lambda_i)}\frac{1}{x-\lambda_i}
\right)
=
-\frac12\frac{P''(\lambda_i)}{P'(\lambda_i)^2}.
\]
Equating both limits yields \eqref{eq:pf-identity}.
\end{proof}

In the following lemma, we compute the function  $\displaystyle\langle X,\nabla \sigma_n\rangle$, and express it in terms of the derivatives of the polynomial $P(x)$
 up to second order. This lemma plays a key role in the proof of Theorem \ref{teo:M^n1-1} and hence Theorem \ref{T-Y}.

\begin{lema}\label{lem-Lr-h} 
Let $\mathfrak{a}$ be a smooth Codazzi symmetric $(0,2)$-tensor field on $D$ satisfying Assumption \ref{assumption-3}.   Suppose that  
its eigenvalues $\lambda_1,\dots,\lambda_n$ are pointwise distinct on $D$. Let \(X=\sum_{i=1}^n\nabla_{e_i}e_{i}\).  If   $\sigma_k=\sigma_k(\lambda_1,\dots,\lambda_n)$ is constant for all $k\neq n$, then
\begin{enumerate}[(i)]
\item  
\(
\Gamma_{ii}^k
=c_{ik}\,(\sigma_n)_k, 
\) for \(i\neq k\),   \(i, k\in \{1, \ldots, n\}\);
    \item $\displaystyle\langle X,\nabla \sigma_n\rangle
=\sum_{k=1}^n v_k\,(\sigma_n)_k^2$.
\end{enumerate}
Here \(\displaystyle c_{ik}=\frac{(-1)^{n+1}}{(\lambda_i-\lambda_k)P'(\lambda_i)}, \text{ for } i\neq k, \text{ and } v_k=(-1)^{n+1}\frac{P''(\lambda_k)}{2P'(\lambda_k)^2}. \)
\end{lema}

\begin{proof} For simplicity of the notation, set \(h=\sigma_n\) and write
\[
\nabla h=\sum_{i=1}^n h_i\,e_i,
\qquad\text{where } h_i=(\sigma_n)_i.
\]
From Proposition \ref{lem:lambda_ij-sigma}, we have
\[
d\lambda_i=\sum_{j=1}^n \lambda_{ij}\,\theta_j,
\qquad 
\lambda_{ij}=(-1)^{\,n+1}\frac{h_j}{P'(\lambda_i)}.
\]
Thus, by Lemma \ref{lem:a_is_Codazzian}, for \(i\neq k\), \ $i, k\in \{1, \ldots, n\}$,
\[
\Gamma_{ii}^k=\frac{\lambda_{ik}}{\lambda_i-\lambda_k}
=\frac{(-1)^{\,n+1}h_k}{(\lambda_i-\lambda_k)P'(\lambda_i)}
=c_{ik}\,h_k,
\]
where \(\displaystyle c_{ik}=\frac{(-1)^{n+1}}{(\lambda_i-\lambda_k)P'(\lambda_i)}, \text{ for } i\neq k.\)

Since \(X=\sum_{i,j}\Gamma_{ii}^j e_j\) and \(\nabla h=\sum_k h_k e_k\),
\[
\langle X,\nabla h\rangle
=\sum_{k=1}^n\big(\sum_{\substack{i=1\\i\neq k}}^n\Gamma_{ii}^k\big)h_k
=\sum_{k=1}^n\big(\sum_{\substack{i=1\\i\neq k}}^n c_{ik}\big)h_k^2
=\sum_{k=1}^n v_k\,h_k^2.
\]
By Proposition  \ref{lem:partial-fraction-identity}, \[v_k=\sum_{\substack{i=1\\i\neq k}}^n c_{ik}=\sum_{\substack{i=1\\i\neq k}}^n\frac{(-1)^{n+1}}{(\lambda_i-\lambda_k)P'(\lambda_i)}=\sum_{\substack{i=1\\i\neq k}}^n\frac{(-1)^{n}}{(\lambda_k-\lambda_i)P'(\lambda_i)}=(-1)^{n+1}\frac12\frac{P''(\lambda_k)}{P'(\lambda_k)^2}.\]
\end{proof}
In Lemma \ref{lem-Lr-h-L}, we  compute the function  $\Psi$, defined in Theorem \ref{thm:divX-intro2} by
\[
\Psi = \sum_{\substack{i<j \\ k\neq i,j}}^n (\Gamma_{ii}^k \Gamma_{jj}^k - \Gamma_{ij}^k \Gamma_{ji}^k),
\]
and apply an  algebraic inequality due to Tang, Wei, and Yan \cite{Tang-Wei-Yan20} to conclude that $\Psi\leq 0$. 
\begin{lema}\label{lem-Lr-h-L} Let $\mathfrak{a}$ be a smooth Codazzi symmetric $(0,2)$-tensor field on $D$ satisfying Assumption \ref{assumption-3}.   Suppose that
its eigenvalues $\lambda_1,\dots,\lambda_n$ are pointwise distinct on $D$.  Let $\Psi$ be as in Theorem \ref{thm:divX-intro2}. If   $\sigma_k=\sigma_k(\lambda_1,\dots,\lambda_n)$ is constant for all $k\neq n$, then
\begin{enumerate}[(i)]
\item \[\Psi
=\frac{1}{2}\sum_{k=1}^n L(k)\,(\sigma_n)_k^2,\]
where 
\(\displaystyle
L(k)=\sum_{\substack{i\neq j\\i,j\neq k}}^n c_{ik}c_{jk}, \) with
\(
c_{ik}
\) defined as in Lemma \ref{lem-Lr-h}.
\item $L(k)<0$,  $k=1,\ldots, n$, and $\Psi\leq 0$.

    \end{enumerate}

\end{lema}

\begin{proof} Set \(h=\sigma_n\). Since $a$ is Codazzian, using  Lemma \ref{lem:a_is_Codazzian} and Lemma \ref{lem-Lr-h}-$(i)$, we have
\[\Psi=\sum_{\substack{i<j\\k\neq i,j}}^n\Gamma_{ii}^k\Gamma_{jj}^k
=\sum_{\substack{i<j\\k\neq i,j}}^nc_{ik}c_{jk}\,h_k^2
=\frac{1}{2}\sum_{k=1}^n
\big(\sum_{\substack{i\neq j\\i,j\neq k}}^nc_{ik}c_{jk}\big)h_k^2
=\frac{1}{2}\sum_{k=1}^n L(k)\,h_k^2.
\]
 An algebraic inequality due to  Tang, Wei, and Yan \cite{Tang-Wei-Yan20} (see the appendix, Lemma \ref{lem-TWY}) implies  that the functions $L(k)<0$, for all $k=1, \ldots, n$. 
Consequently, $\Psi\leq 0$. 
\end{proof}

\section{Proof of Theorem \ref{T-Y}}\label{sec:5}

In this section and Section \ref{sec:6}, we assume that $M^n$ is connected and oriented. Otherwise, the arguments can be carried out on each connected component or on the oriented double cover.

A special case of Theorem \ref{T-Y} is the following result proved by Tang,  Wei and Yan \cite{Tang-Wei-Yan20}.  
\begin{theorem}[Tang--Wei--Yan \cite{Tang-Wei-Yan20}]\label{T-W-Y}
Let $M^n$ $(n>3)$ be a closed Riemannian manifold.  
If $\mathfrak{a}$ is a smooth symmetric tensor field of type $(0,2)$ on $M$ and $A$ is its associated
$(1,1)$-tensor field. Assume  that:
\begin{enumerate}
    \item $S_{M} \geq 0$;
    \item $\mathfrak{a}$ is Codazzian;
    \item $A$ has $n$ distinct eigenvalues everywhere;
    \item $\operatorname{tr}(A^k)$ is constant for $k=1,\dots,n-1$.
    \end{enumerate}
Then  $\operatorname{tr}(A^{n})$ is constant and thus all eigenvalues of $A$ are constant, and $S_M\equiv 0$.  

\end{theorem}
We give a proof of Theorem \ref{T-W-Y}.
\begin{proof} Without loss of generality, we label the $n$ eigenvalues of $A$ by $\lambda_1<\lambda_2<\ldots<\lambda_n$ on $M$. Then the eigenvalue functions $\lambda_i$  are smooth on $M$, and locally one can choose a smooth orthonormal frame of the corresponding eigenvectors $\{e_i\}$  on $M$. 

 Notice  that $X=\sum_{i=1}^n\nabla_{e_i}e_{i}$ and  $\Psi$ are globally well-defined on $M$. In fact, for any two local orthonormal frames $(U,\{e_i\})$ and $(\tilde{U}, \{\tilde{e_i}\})$ with $U\cap\tilde{U}\neq \emptyset$, we have $e_i=\pm\tilde{e_i}$. Then $X=\sum_{i=1}^n\nabla_{e_i}e_i=\sum_{i=1}^n\nabla_{\tilde{e_i}}\tilde{e_i}$ on $U\cap\tilde{U}$.   

Observe that for any $l\in\{1,\ldots, n\}$, the values of  $tr(A^k)$, $k=1,\ldots, l$ and the values of $\sigma_k(A)$, $k=1,\dots,l$ determine each other. Hence Assumption $(4)$ is equivalent to the constancy of $\sigma_k(A)$ for $k = 1, \ldots, n-1$. 

By Lemma \ref{lem-Lr-h-L},  $\Psi\leq 0$ on $M$. By Theorem \ref{thm:divX-intro2}, $\dv(X)=\frac12S_M-\Psi\geq0$ on $M$. 

On the other hand, since $M$ is closed, we have
\[\int_M\dv(X)=0.\]
It follows that $\dv(X)\equiv 0$, and consequently $S_M\equiv 0$ and $\Psi\equiv 0$ on $M$. This, together with Lemma \ref{lem-Lr-h-L}, implies that $\nabla \sigma_n\equiv 0$ and hence $\sigma_n$ is constant. Therefore $\operatorname{tr}(A^n)$ is constant and thus all eigenvalues of $A$ are constant.
\end{proof}

The following lemma is a corollary of a result of Almeida–Brito \cite{Almeida-Brito90} and will be used in the proof of Theorem \ref{teo:M^n1-1}.

 \begin{lema}
\label{lema:boundary}
Let $M$ be a closed Riemannian manifold. Suppose $u:M\to\mathbb{R}$ is a smooth function and let $m=\min_M u$.
If $a_k>m$ are regular values of $u$ with $a_k\to m$ and
\[
D_k=u^{-1}([m,a_k]),
\]
then
\begin{equation}
\label{eq:boundary}
\lim_{k\to \infty}\int_{\partial D_k}|\nabla u|\,dS=0.
\end{equation}
\end{lema}

\begin{proof}
For each regular value $a_k$, the boundary $\partial D_k=u^{-1}(a_k)$ is a smooth hypersurface whose outward unit normal vector is
\[
\nu=\frac{\nabla u}{|\nabla u|}.
\]
By the divergence theorem,
\[
\int_{\partial D_k}|\nabla u|\,dS
= \int_{\partial D_k}\langle\nabla u,\nu\rangle\,dS
= \int_{D_k}\mathrm{div}(\nabla u)\,d\mathrm{vol}
= \int_{D_k}\Delta u\,d\mathrm{vol}.
\]
Taking absolute values gives
\[
\int_{\partial D_k}|\nabla u|
\le \int_{D_k}|\Delta u|\,d\mathrm{vol}.
\]
Since $a_k-m\to0$, 
by Lemma~\ref{lema:advanced} in the Appendix,
\[
\lim_{k\to\infty}\int_{D_k}|\Delta u|\,d\mathrm{vol}=0.
\]
Hence  \eqref{eq:boundary} follows.
\end{proof}

To prove Theorem \ref{T-Y}, we establish the following result.
\begin{thm}
\label{teo:M^n1-1}
Let $M^n,$ $(n\ge 3)$ be a closed $n$-dimensional Riemannian manifold. 
Suppose that $\mathfrak{a}$ is a smooth  Codazzi symmetric $(0,2)$-tensor field on $M^n$ and $A$ its associated $(1,1)$-tensor field.
Assume that  $\sigma_k=\sigma_k(A)$, $k=1,\dots,n-1$, are constant and $\sigma_n=\sigma_n(A)$ is not constant on $M$. Then:
\begin{enumerate}[(i)]
    \item The image of $(-1)^n\sigma_n$ is a compact interval $[a,b]$ with  $a<b$.
    \item Let $D_0=\{p\in M\mid (-1)^n\sigma_n(p)\in (a,b)\} $.   For each $p\in D_0$, $A$ has $n$ distinct eigenvalues   \[\lambda_1(p)<\lambda_2(p)<\cdots<\lambda_n(p). \] 
    Moreover, the eigenvalues $\lambda_i $  are smooth on $D_0$ and locally there exist smooth orthonormal eigenvectors  $e_i$ associated with $\lambda_i$, $i=1, \ldots, n$.  
   \item Let $e_i, i=1,\ldots, n$, be as in (ii). Then the vector field  \[X=\sum_{i=1}^n\nabla_{e_i}e_{i}\] is well-defined on $D_0$.  
    \item There exist two sequences $\{a_k\} \text{ and } \{ b_k\}$ with $[a_k, b_k]\subset (a,b)$,  $\displaystyle\lim_{k\to +\infty}a_k=a$, and $\displaystyle\lim_{k\to +\infty}b_k= b$, such that 
\[
\limsup_{k\to +\infty} \int_{D_k} \mathrm{div}(X)\le 0, \]
where $D_k:=\{p\in M; (-1)^n\sigma_n(p)\in (a_k,b_k) \}$ and $X$ is defined in (iii).
\end{enumerate}

\end{thm}
\begin{proof}
$(i)$ Since  $\sigma_n$ is a smooth, non-constant function on the closed manifold $M^n$,  the image of $(-1)^{n}\sigma_n$ is a compact interval $[a,b]$ with $a<b$.

$(ii)$ Denote by $\lambda_1(p)\le\cdots\le\lambda_n(p)$ the eigenvalues of $A(p)$ for  $p\in M^n$. It is known that $\lambda_i$, $i=1,\ldots, n$, are continuous on $M$ and there exists an open dense subset $V\subset M $ on which all $\lambda_i$ are smooth (see Remark \ref{Re-S75}). 
Define the open subset of $M$ by
\[
D:=\big\{p\in M^n \mid \lambda_1(p)<\lambda_2(p)<\cdots<\lambda_n(p)\big\}.
\]
\begin{claim}\label{claim-5-1}
$D$ is nonempty.
\end{claim}

In fact, assume, for the sake of contradiction,  that $D=\emptyset$.

For $p\in M$, the characteristic polynomial of $A$  is
\begin{equation}
\label{eq:P(lambda)-sigma-1}
P(x)=\prod_{i=1}^n(x-\lambda_i)
=x^n-\sigma_1x^{n-1}+\sigma_2x^{n-2}-\cdots+(-1)^{n-1}\sigma_{n-1}x+(-1)^n\sigma_n,
\end{equation}
where $\sigma_1,\dots,\sigma_{n-1}$ are constant on $M$ and $\sigma_n$ is not constant on $M$.

 Differentiating  the equation $P(\lambda_i)=0$ on $V$, we obtain 
\begin{equation}\label{eq:fund-diff}
P'(\lambda_i)\,d\lambda_i + (-1)^n d\sigma_n = 0.
\end{equation}

For any  $p\in V$, since $D=\emptyset$,  there exists an index \(i\) such that the eigenvalue \(\lambda_i(p)\) has multiplicity at least $ 2$. For this index  $i$, we have
\(
P'(\lambda_i(p))=0.
\) Substituting in \eqref{eq:fund-diff},  we have \(d\sigma_n(p)=0\).

Since $V$ is an  open dense subset of \(M^n\), it follows that \(d\sigma_n\equiv 0\) on \(M\) and thus \(\sigma_n\) is constant on \(M\), which contradicts the assumption of the theorem. Therefore,  Claim \ref{claim-5-1} holds.

 We define the  polynomial $Q(x)$ of degree $n$ by
\[Q(x)=x^n-\sigma_1x^{n-1}+\sigma_2x^{n-2}-\cdots+(-1)^{n-1}\sigma_{n-1}x.\]
Then
\[P(x)=Q(x)+(-1)^{n}\sigma_n.\] 
 Consequently,
\[Q(0)=0, \  Q'(x)=P'(x)=\sum_{j=1}^n\prod_{i\neq j}^{n}(x-\lambda_i), \ \text{ and }  Q''(x)=P''(x).\]

For each $p\in M$, since the term $(-1)^{n}\sigma_n$ is independent of $x$,  the graph of $P(x)$ is merely a vertical translation of the graph of $Q(x)$ by  $(-1)^{n}\sigma_n$. Thus, knowing the graph of $P(x)$ for  $p\in D$  determines that of $Q(x)$, and hence the graphs of $P(x)$ for all $p\in M$.

Choose a point $p\in D\neq \emptyset$. Since $P$ has $n$ distinct real roots at $p$, $P$ has $n-1$ distinct critical points $\xi_1<\cdots<\xi_{n-1}$, which are  local extrema that alternate between maxima and minima, satisfying 
\begin{equation}\label{ine-xi}
 \lambda_1<\xi_1<\lambda_2<\ldots<\lambda_i<\xi_i<\lambda_{i+1}<\ldots<\lambda_{n-1}<\xi_{n-1}<\lambda_n.   
\end{equation}
Consequently, $Q(x)$ has the same $n-1$ distinct critical points $\xi_1<\cdots<\xi_{n-1}$, again with alternating local extrema. Hence  $\displaystyle Q'(x)=n\prod_{i=1}^{n-1}(x-\xi_i)$, $Q'(\xi_i)=0$, and  $Q''(\xi_i)\neq 0$ by observing that $\xi_i$ are distinct.

We define two numbers $\underline{M}$ and $\overline{m}$ which are  the smallest local maximum value of $Q$ and  the largest local minimum value of $Q$ respectively, that is,
\[
\underline{M} := \min\{Q(\xi_i): Q''(\xi_i)<0\}, \qquad
\overline{m} := \max\{Q(\xi_i): Q''(\xi_i)>0\}.
\]
 
Noting that the values of the $n$ real roots $\lambda_i$ of $P(x)$ depend on the value of $(-1)^n\sigma_n$, we may also regard $P(x)$ as a one-parameter family of polynomials of degree $n$, parametrized by $t$, 
\begin{equation}\label{eq:t}
    P(x)(t)=Q(x)+t, \  t\in [a,b].
\end{equation}

There is  a relation between  the properties of the roots of $P(x)$ and the quantities \(\underline{M}\)  and \(\overline{m} \), which can be described as follows:

(a) The  existence of $n$ real roots $\lambda_i$ of the polynomial $P$ requires that
\[t\in [-\underline{M},-\overline{m}],\]
that is, 
\begin{equation*}
  [a,b]\subset[-\underline{M},-\overline{m}]. 
\end{equation*}

(b) By the shape of graph of $P$,  for each $t\in [a,b]$, every real root $\lambda_i$ of $P$ has multiplicity  at most $2$.  

Moreover,  $P$ has $n$ distinct real roots if and only if \[t\neq  -\underline{M} \text{ and } t\neq -\overline{m}.\]
It follows that  for $t\in (a,b)\subset  (-\underline{M},-\overline{m})$, $P$ has $n$ distinct real roots $\lambda_i$, which satisfy \[\lambda_1<\xi_1<\lambda_2<\ldots<\lambda_i<\xi_i<\lambda_{i+1}<\ldots<\lambda_{n-1}<\xi_{n-1}<\lambda_n,\]
where $\xi_i$, $i=1,\ldots,n-1$ are defined in (\ref{ine-xi}).

This implies that \[D_0=\{p\in M\mid (-1)^n\sigma_n(p)\in (a,b)\}\subset D\]

(c) If the case $ a=-\underline{M}$  (or $b=-\overline{m}$, respectively) occurs, then at $t=a$ (or $t=b$, respectively)  there exists at least one index $i$ such that $\lambda_i=\lambda_{i+1}=\xi_i$, where $\xi_i$ satisfies
\[
P(\xi_i)=0,   P'(\xi_i)=0, \text{ and }  P''(\xi_i)<0 \ (\text{or } P''(\xi_i)>0, \text{ respectively}).
\]

By (b),  $\lambda_1 < \lambda_2 < \cdots < \lambda_n$ on $D_0$. This implies that the eigenvalue functions $\lambda_i$ are smooth on $D_0$, and one can choose a  local smooth orthonormal frame of the corresponding eigenvectors $\{e_i\}$ on $D_0$.

$(iii)$ By the same argument as in the proof of Theorem \ref{T-W-Y}, $X=\sum_{i=1}^n\nabla_{e_i}e_i$ is globally well-defined on $D_0$.

$(iv)$ Since $(-1)^n\sigma_n$ is smooth on $M$, by Sard's lemma, we may choose two sequences of regular values $a_k$ and $b_k$ of $(-1)^n\sigma_n$ such that $[a_k, b_k]\subset(a,b)$, $a_k\to a$ and $b_k\to b$.  By the Stokes theorem, we have
\begin{equation}\label{eq:stokes}
 \int_{D_{k}} \mathrm{div}(X)=\int_{\partial D_k}\langle X,\nu\rangle,   
\end{equation}
where $D_k=\{p\in M; (-1)^n\sigma_n(p)\in(a_k, b_k)\} $ and $\nu$ is the unit outward normal vector to $\partial D_k$.  

Let $\partial D_k^a=\{p\in M; (-1)^n\sigma_n(p)=a_k\}$ and  $\partial D_k^b=\{p\in M; (-1)^n\sigma_n(p)=b_k\}.$ Then 
\[\partial D_k=\partial D_k^a\cup \partial D_k^b.\]

Since  $\nu=\frac{(-1)^n\nabla \sigma_n}{|\nabla \sigma_n|}$ on $\partial D_k^b$, by Lemma  \ref{lem-Lr-h}, 
we have, on $\partial D_k^b$,
\begin{equation}\label{eq:ak}
  \begin{split}
\langle X,\nu\rangle
&=\frac{(-1)^n}{|\nabla \sigma_n|}\langle X,\nabla \sigma_n\rangle =\frac{(-1)^n}{|\nabla \sigma_n|}\sum_{i=1}^nv_i(\sigma_n)_i^2=-\frac12\sum_{i=1}^n\frac{(\sigma_n)_i^2}{|\nabla \sigma_n|}\frac{P''(\lambda_i)}{P'(\lambda_i)^2}.
\end{split}  
\end{equation}
Similarly, on $\partial D_k^a$,  $\nu=-\frac{(-1)^n\nabla h}{|\nabla h|}$ and
\begin{equation}\label{eq:bk}
 \begin{split}
\langle X,\nu\rangle
&=-\frac{(-1)^n}{|\nabla \sigma_n|}\langle X,\nabla \sigma_n\rangle =-\frac{(-1)^n}{|\nabla \sigma_n|}\sum_{i=1}^nv_i(\sigma_n)_i^2=\frac12\sum_{i=1}^n\frac{(\sigma_n)_i^2}{|\nabla \sigma_n|}\frac{P''(\lambda_i)}{(P'(\lambda_i))^2}.
\end{split}   
\end{equation}

We next prove the  two claims for the one-parameter family of polynomials $P(x)$ with parameter $t$, given by \eqref{eq:t}.

\begin{claim}\label{claim-5-2}
    $\tfrac{P''(\lambda_i)}{P'(\lambda_i)^2}$, $i=1,\ldots,n$,  are bounded from above on $(a, a+\delta]$  for some $\delta>0$.
\end{claim}

We prove the claim by considering two possible cases for $a$.

Case 1. $a>-\underline{M}$. In this case,  $\lambda_i(a)$, $i=1,\ldots,n$,  are distinct. Hence, as  $t\searrow a$,  $\lambda_i\rightarrow \lambda_i(a)$, and
\[
\frac{P''(\lambda_i)}{P'(\lambda_i)^2}\to  \frac{P''(\lambda_{i}(a))}{P'(\lambda_{i}(a))^2}.
\]
It implies that 
$\tfrac{P''(\lambda_i)}{P'(\lambda_i)^2}$, $i=1,\ldots,n$,  are bounded on $(a, a+\delta]$ for some $\delta>0$.

Case 2. $a=-\underline{M}$.

 When $n$ is odd, the roots $\lambda_i(a)$ satisfy that
\begin{equation}\label{ine:odd-1}
    \lambda_1(a)\leq \lambda_2(a)<\lambda_3(a)\leq \lambda_4(a)<\lambda_5(a)\leq\ldots<\lambda_{n-2}(a)\leq\lambda_{n-1}(a)<\lambda_n(a),
\end{equation}
where the equality in $ \lambda_i(a)\leq\lambda_{i+1}(a)$  in \eqref{ine:odd-1}  occurs only for those indices $i$ such that $Q(\xi_i)=\underline{M}$. Moreover, for such an index $i$, one has $\lambda_i(a)=\lambda_{i+1}(a)=\xi_i$.

\begin{figure}[htbp]\label{fig1}
  \centering
  \includegraphics[width=0.5\textwidth]{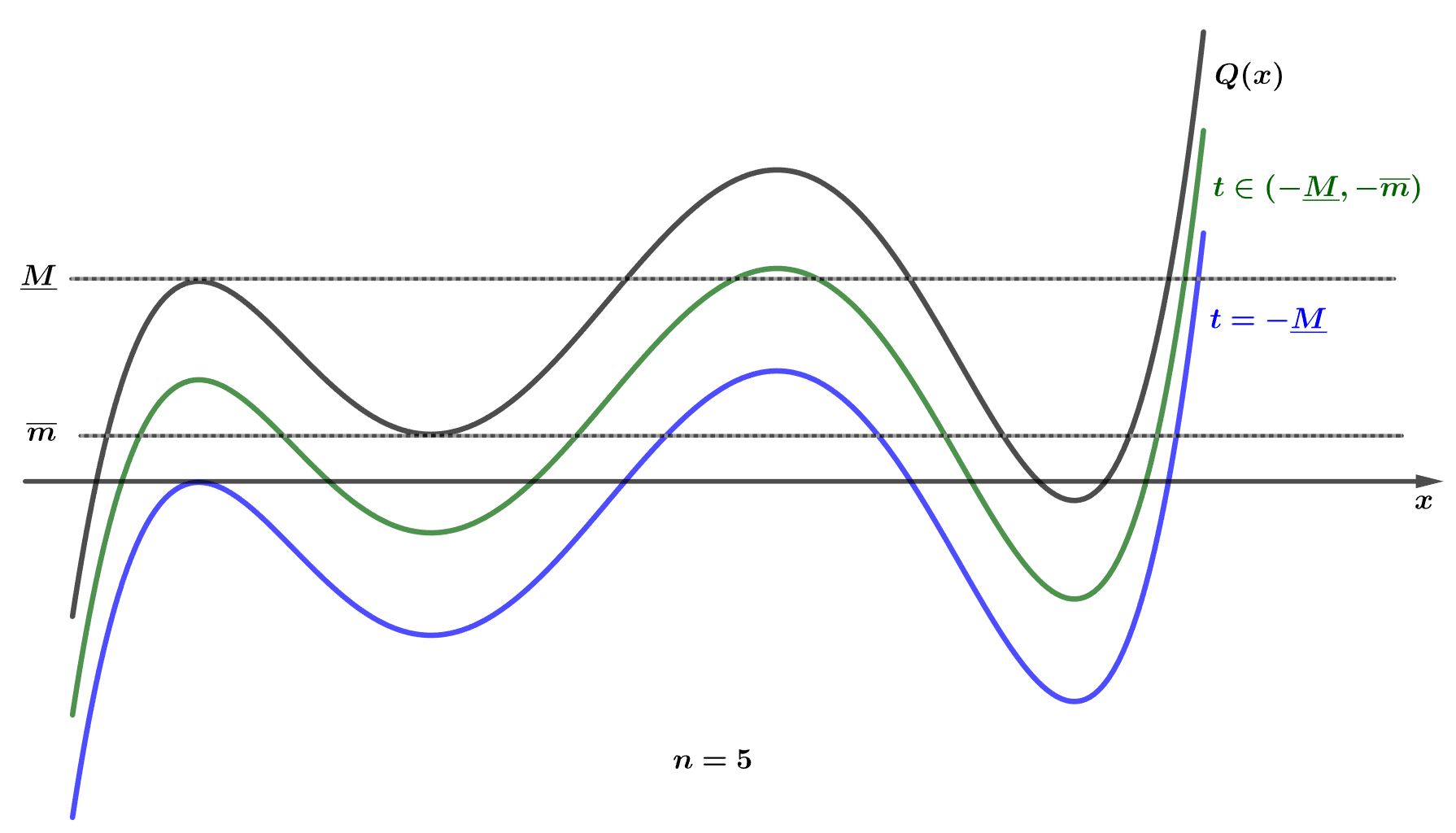}
  \caption{An example illustrating the case of equality in \eqref{ine:odd-1}}
  \label{fig:figure1}
\end{figure}

Now for the index $i$ such that $\lambda_i(a)=\lambda_{i+1}(a)=\xi_i$, we have $Q(\xi_i)=\underline{M}$, $Q'(\xi_i)=0$, and $Q''(\xi_i)<0$. Hence, as $t\searrow a$,  $\lambda_i\rightarrow \xi_i$ and $\lambda_{i
+1}\rightarrow \xi_i$. Consequently,
\begin{equation}\label{odd-1}
 \frac{P''(\lambda_i)}{P'(\lambda_i)^2}\to -\infty \ \text{ and } \ \frac{P''(\lambda_{i+1})}{P'(\lambda_{i+1})^2}\to -\infty.   
\end{equation}

For the index $i$ such that $Q(\xi_i)> \underline{M}$, as  $t\searrow a$,  $\lambda_i\rightarrow \lambda_i(a)$,  $\lambda_{i+1}\rightarrow \lambda_{i+1}(a)$. Hence,

\begin{equation}\label{odd-2}
   \frac{P''(\lambda_i)}{P'(\lambda_i)^2}\to  \frac{P''(\lambda_{i}(a))}{P'(\lambda_{i}(a))^2} \ \text{ and }\  \frac{P''(\lambda_{i+1})}{P'(\lambda_{i+1})^2}\to \frac{P''(\lambda_{i+1}(a))}{P'(\lambda_{i+1}(a))^2}. 
\end{equation}

By \eqref{ine:odd-1} and the alternating nature of the local extrema $\xi_i$, we observe that for each index $i\neq n$, $\lambda_i=\lambda_i(t)$, with $ t\in (a,b),$ coincides either with the left endpoint  of the interval $[\lambda_i, \lambda_{i+1}]$ such that $Q(\xi_i)\geq \underline{M}$, or with the right endpoint  of the interval $[\lambda_{i-1}, \lambda_{i}]$ such that $Q(\xi_{i-1})\geq \underline{M}$. Hence \eqref{odd-1} and \eqref{odd-2} imply that $\tfrac{P''(\lambda_i)}{P'(\lambda_i)^2}$, $i\neq n$,  are bounded from above on $(a, a+\delta]$  for some $\delta>0$.

On the other hand, $\lambda_n$ is the right endpoint of $[\lambda_{n-1},\lambda_n]$ with $Q(\xi_{n-1})\leq \overline{m}<\underline{M}$. Hence as $t\searrow a$, $\lambda_n\to \lambda_n(a)$, and
\begin{equation}\label{odd-3}
   \frac{P''(\lambda_n)}{P'(\lambda_n)^2}\to  \frac{P''(\lambda_{n}(a))}{P'(\lambda_{n}(a))^2}. 
\end{equation}
Hence  $\tfrac{P''(\lambda_n)}{P'(\lambda_n)^2}$ is bounded  on $(a, a+\delta]$  for some $\delta>0$.
 
When $n$ is even,  the following observations hold: 

If $a=-\underline{M}$, the roots $\lambda_i(a)=\lambda(-\underline{M})$ satisfy that
\begin{equation}\label{ine:even-1}
    \lambda_1(a)< \lambda_2(a)\leq\lambda_3(a)< \lambda_4(a)\leq\lambda_5(a)<\ldots<\lambda_{n-2}(a)\leq \lambda_{n-1}(a)<\lambda_n(a),
\end{equation}
where the equality in $ \lambda_i(a)\leq\lambda_{i+1}(a)$  in \eqref{ine:even-1}  occurs only for those indices $i$ such that $Q(\xi_i)=\underline{M}$. Moreover, for such an index $i$, one has $\lambda_i(a)=\lambda_{i+1}(a)=\xi_i$.

From \eqref{ine:even-1} and the alternating nature of the local extrema $\xi_i$,  it follows  that each $\lambda_i=\lambda_i(t)$, with $ t\in (a,b)$, $i\neq 1, n$, coincides either with the left endpoint  of the interval $[\lambda_i, \lambda_{i+1}]$ such that $Q(\xi_i)\geq \underline{M}$, or with the right endpoint  of the interval $[\lambda_{i-1}, \lambda_{i}]$ such that $Q(\xi_{i-1})\geq \underline{M}$.  

On the other hand, $\lambda_1$ is the left endpoint of $[\lambda_1,\lambda_2]$ with $Q(\xi_1)\leq \overline{m}$ and  $\lambda_n$ is the right endpoint of $[\lambda_{n-1},\lambda_n]$ with $Q(\xi_{n-1})\leq \overline{m}$.

\begin{figure}[htbp]\label{fig2}
  \centering
  \includegraphics[width=0.6\textwidth]{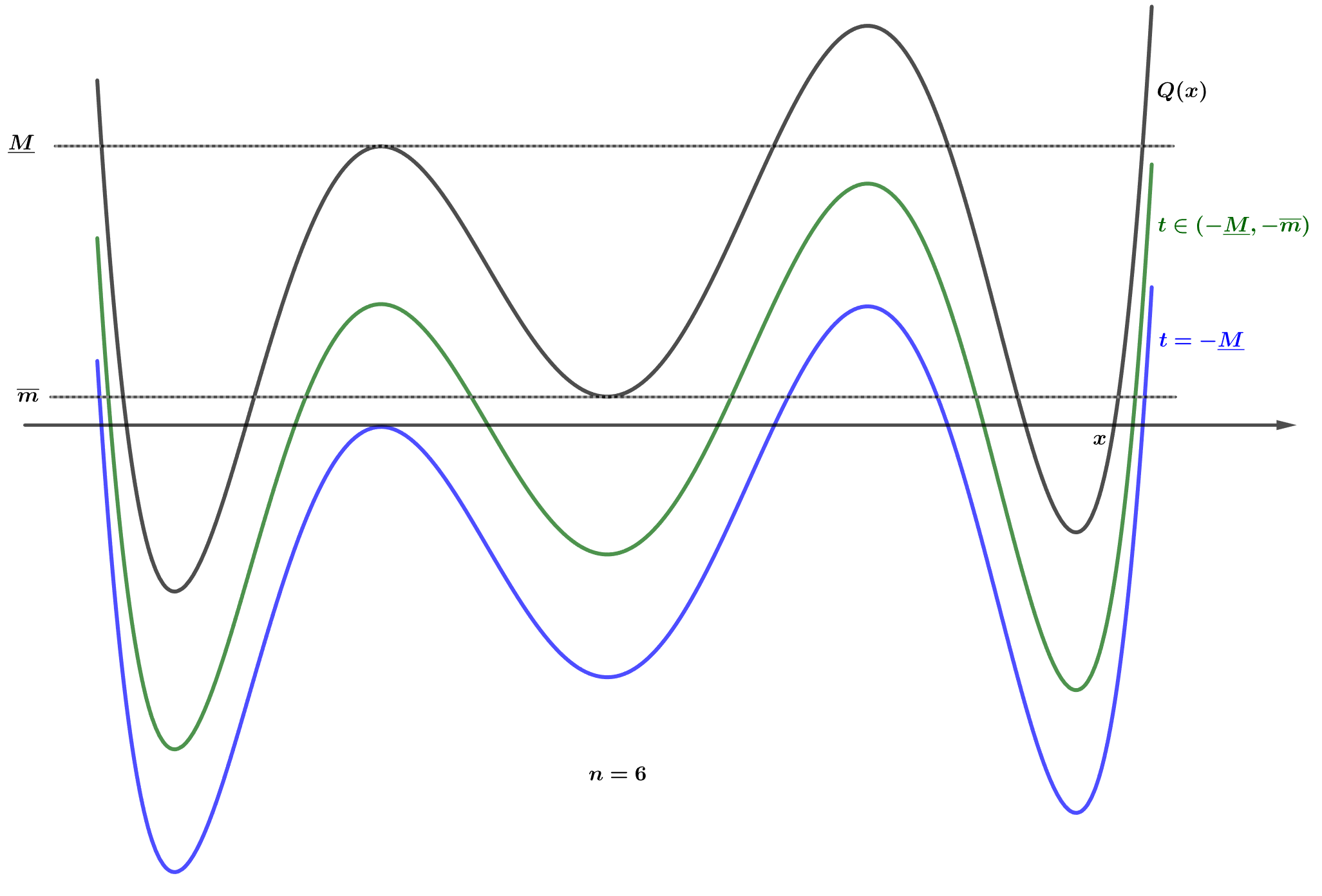}
  \caption{An example illustrating the case of equality in \eqref{ine:even-1}}
  \label{fig:figure2}
\end{figure}

Using the above observations and the same argument as in the case of  odd $n$, we can prove that $\tfrac{P''(\lambda_i)}{P'(\lambda_i)^2}$, $i=1,\ldots,n$,  are bounded from above on $(a, a+\delta]$  for some $\delta>0$.

Therefore Claim \ref{claim-5-2} holds.

\begin{claim}\label{claim-5-3}
  $\tfrac{P''(\lambda_i)}{P'(\lambda_i)^2}$, $i=1,\ldots,n$,  are bounded from below on $[b-\delta,b)$ for some $\delta>0$.  
\end{claim}   

We omit the proof, which is analogous to that of Claim \ref{claim-5-2}, and only point out the key property used:

In the case  $b=-\overline{m}$, for any index $i$ such that  $Q(\xi_i)=\overline{m} $, as $t\nearrow b$, we have  $\lambda_i\rightarrow \xi_i$ and $\lambda_{i+1}\rightarrow \xi_i$.  Since $Q'(\xi_i)=0$ and $Q''(\xi_i)>0$, it follows that
\[
\frac{P''(\lambda_i)}{P'(\lambda_i)^2}\to +\infty \ \text{ and }\  \frac{P''(\lambda_{i+1})}{P'(\lambda_{i+1})^2}\to +\infty.
\]

Now we are ready to prove the conclusion of the item. 

By (\ref{eq:ak}), Claim \ref{claim-5-2},  and Lemma \ref{lema:boundary}, 
for $k$ sufficiently large, 
\begin{equation}\label{ine:sup3}
    \int_{\partial D^a_k} \langle X,\nu\rangle \le C \int_{\partial D^a_k} |\nabla h| \to 0,  \quad \text{as } k \to \infty.
\end{equation}
Similarly, by \eqref{eq:bk}, Claim \ref{claim-5-3}, and Lemma \ref{lema:boundary},  for $k$ sufficiently large,
\begin{equation}\label{ine:sup4}
  \int_{\partial D^b_k} \langle X,\nu\rangle \le C \int_{\partial D^b_k} |\nabla h|,  \quad \text{as } k \to \infty.  
\end{equation}
Using  \eqref{eq:stokes}, together by \eqref{ine:sup3} and \eqref{ine:sup4}, we obtain
\[
\limsup_{k\to +\infty}\int_{D_k} \mathrm{div}(X) \leq\limsup_{k\to +\infty} \int_{\partial D^a_k} \langle X,\nu\rangle +\limsup_{k\to +\infty} \int_{\partial D^b_k} \langle X,\nu\rangle=0.
\]
This completes the proof of Theorem \ref{teo:M^n1-1}.
\end{proof}
We now prove Theorem \ref{T-Y}  proved in \cite{Tang-Yan23} by  Tang and Yan. Notice that for any $l\in\{1,\ldots, n\}$, the values of  $tr(A^k)$, $k=1,\ldots, l$ and the values of $\sigma_k(A)$, $k=1,\dots,l$ are determined each other. Consequently,
Theorem \ref{teo:M^n} is equivalent to  Theorem \ref{T-Y}. 
\begin{thm} [alternate version of Theorem \ref{T-Y}]
\label{teo:M^n}
Let $M^n$ $(n>3)$ be a closed Riemannian manifold.  
Suppose that  $\mathfrak{a}$ is a smooth symmetric tensor field of type $(0,2)$ on $M$ and $A$ is its associated
$(1,1)$-tensor field. Assume  that:
\begin{enumerate}
    \item $S_{M} \geq 0$;
    \item $\mathfrak{a}$ is Codazzian;
    \item the elementary symmetric functions $\sigma_k(A)$, $k=1,\dots,n-1$, are constant.
\end{enumerate}
Then $\sigma_n(A)$ is also constant, and hence all eigenvalues of $A$ are constant on $M$. Moreover, if $A$ has $n$ distinct eigenvalues somewhere on $M^n$, then $S_M \equiv 0$ on $M$.
\end{thm}

\begin{proof}
We assume for the sake of contradiction that $\sigma_n(A)$  is not constant. By Lemma \ref{lem-Lr-h-L}, $\Psi\leq 0$ on $D_0$. By $S_M\ge 0$  and  (\ref{div(X)}):
$\dv (X) =\frac12 S_M - \Psi$, it follows that $\mathrm{div}(X)\ge 0$ on $D_0$. From (iv) of Theorem \ref{teo:M^n1-1}, we have $\lim_{k\to +\infty}\int_{D_k}\dv(X)=0$ and then $\int_{D_k} \mathrm{div}(X)=0$ for all $k$. Hence $\dv(X)\equiv 0$ and then $\Psi\equiv 0$ on $D_0$.  By (i) of Lemma  \ref{lem-Lr-h-L}, $(\sigma_n)_i\equiv 0$, $i=1, \ldots, n,$ on $D_0$. Consequently,   $\sigma_n$ is constant on $D_0$ and by continuity it is constant on $M$, which yields a contradiction.

Therefore, $\sigma_n(A)$ is constant on $M$, and hence all eigenvalues of $A$ are constant on $M$.

If  $A$ has $n$ distinct eigenvalues somewhere on $M$, then all eigenvalues of $A$ are distinct constants on $M$. By Lemma  \ref{lem-Lr-h-L}, $\Psi\equiv 0$ on $M$. Hence 
\[\int_M\frac12S_M=\int_M\dv(X)=0.
\]
Since $S_M\geq 0$, we have $S_M\equiv 0$ on $M$. This completes the proof of Theorem~\ref{teo:M^n}.
\end{proof}

\begin{remark}[Smoothness of eigenvalues and eigenvectors]\label{Re-S75}
In applications, the orthonormal base can be chosen as the eigenvectors of a smooth symmetric $(0,2)$ tensor. The smoothness of the eigenvalues and eigenvectors of a smooth symmetric $(0,2)$ tensor field
was established by D.~H.~Singley in his work \cite{Singley75}.
Let $M^n$ be a smooth manifold endowed with two symmetric covariant tensor fields $G$ and $G'$, 
where $G$ is positive definite. 
Considering the $(1,1)$ tensor $A = G^{-1}G'$, Singley proved that the eigenvalues of $A$ 
(identified as the principal curvatures in the hypersurface case) are continuous on $M$
and smooth on an open dense subset of $M$. 
Moreover, in a neighborhood of any point in this open dense subset, one can choose the corresponding eigenvector fields to be smooth.

\end{remark}

\section{Vanishing determinant case} \label{sec:6}

In this section, we  prove  Theorem \ref{teo:M^n-(n-1)-intro2} using the same approach as in the proof of Theorem \ref{teo:M^n}.

\begin{lema}\label{lem:lambda_ij-sigma-variant} Let $D$ be an open subset  of an $n$-dimensional Riemannian manifold $M^n$ with an orthonormal frame $\{e_i\}$ and its dual coframe  $\{\theta_i\}$. 
Assume that $\lambda_1,\dots,\lambda_n$ are smooth  pointwise distinct real functions on  $D$. 
Let $\sigma_k=\sigma_k(\lambda_1,\dots,\lambda_n)$ denote the $k$-th elementary symmetric function, so that
\[
P(x):=\prod_{i=1}^n (x-\lambda_i)
= x^n - \sigma_1 x^{\,n-1} + \sigma_2 x^{\,n-2} - \cdots + (-1)^n \sigma_n.
\]
Let \(
d\lambda_i=\sum_{j=1}^n \lambda_{ij}\,\theta_j.
\)
If $\sigma_k$ is constant for all $k\neq n-1$, then 
\[
\lambda_{ij} \;=\; (-1)^{\,n}\,\frac{\lambda_i\,(\sigma_{n-1})_j}{P'(\lambda_i)}, \qquad i,j=1,\dots,n,
\]
where
\[
d\sigma_{n-1}=\sum_{j=1}^n (\sigma_{n-1})_j\,\theta_j, \text{ and }
P'(\lambda_i)=\prod_{k\neq i}^{n}(\lambda_i-\lambda_k).
\]
\end{lema}

\begin{proof}
Differentiating $P(\lambda_i)=0$ and using the hypothesis $d\sigma_k=0$ for every $k\neq n-1$,  we have
\[
P'(\lambda_i)\,d\lambda_i + (-1)^{\,n-1}\lambda_i\,d\sigma_{n-1}=0.
\]
Solving for \(d\lambda_i\) gives
\[
d\lambda_i 
= (-1)^n\frac{\lambda_i}{P'(\lambda_i)}d\sigma_{n-1},
\]
which yields the component formula
\[
\lambda_{ij}=(-1)^n\frac{\lambda_i\,(\sigma_{n-1})_j}{P'(\lambda_i)}.
\]

\end{proof}

\begin{prop}\label{lem-Lr-h-(n-1)} Let $\mathfrak{a}$ be a smooth Codazzi symmetric $(0,2)$-tensor field on $D$ satisfying Assumption \ref{assumption-3}. Let \(X=\sum_{i=1}^n\nabla_{e_i}e_{i}\) and \(\Psi\) be as in Theorem \ref{thm:divX-intro2}. Suppose that its eigenvalues $\lambda_i$  are pointwise distinct on $D$.  
If $\sigma_k=\sigma_k(\lambda_1,\dots,\lambda_n)$ is constant for all $ k\neq n-1$,
then:  
\begin{enumerate}[(i)]
\item $\Gamma_{ii}^k
=b_{ik} (\sigma_{n-1})_k$, for \(i\neq k\),  \ $i, k\in \{1, \ldots, n\}$;
\item $\displaystyle\langle X,\nabla \sigma_{n-1}\rangle
=\sum_{k=1}^n u_k\,(\sigma_{n-1})_k^2$;
\item $\displaystyle\Psi
=\frac{1}{2}\sum_{k=1}^n G(k)\,(\sigma_{n-1})_k^2$.
    \end{enumerate}
Here \(\displaystyle
b_{ik}=\frac{(-1)^{n}\lambda_i}{(\lambda_i-\lambda_k)P'(\lambda_i)}, \text{ for } i\neq k, \quad    u_k=\sum_{\substack{i=1\\i\neq k}}^n b_{ik},\quad 
G(k)=\sum_{\substack{i\neq j\\i,j\neq k}}^n b_{ik}b_{jk}.
\)
\end{prop}

\begin{proof}  Set \(h=\sigma_{n-1}\).
From Lemma~\ref{lem:lambda_ij-sigma-variant}, we have that
\[
d\lambda_i=\sum_{j=1}^n \lambda_{ij}\,\theta_j,
\qquad 
\lambda_{ij}=(-1)^{\,n}\frac{\lambda_ih_j}{P'(\lambda_i)}.
\]
Thus, by Lemma \ref{lem:a_is_Codazzian}, for \(i\neq k\),
\[
\Gamma_{ii}^k=\frac{\lambda_{ik}}{\lambda_i-\lambda_k}
=\frac{(-1)^{\,n}\lambda_i h_k}{(\lambda_i-\lambda_k)P'(\lambda_i)}
=b_{ik}\,h_k.
\]
Since \(X=\sum_{i,j}\Gamma_{ii}^j e_j\) and \(\nabla h=\sum_k h_k e_k\),
\[
\langle X,\nabla h\rangle
=\sum_{k=1}^n\big(\sum_{\substack{i=1\\i\neq k}}^n\Gamma_{ii}^k\big)h_k
=\sum_{k=1}^n\big(\sum_{\substack{i=1\\i\neq k}}^n b_{ik}\big)h_k^2
=\sum_{k=1}^n u_k\,h_k^2.
\]
Since $a$ is Codazzian, using  Lemma \ref{lem:a_is_Codazzian}, it follows that
\[
\Psi=\sum_{\substack{i<j\\k\neq i,j}}^{n}\Gamma_{ii}^k\Gamma_{jj}^k
=\sum_{\substack{i<j\\k\neq i,j}}^{n}b_{ik}b_{jk}\,h_k^2
=\frac{1}{2}\sum_{k=1}^n
\big(\sum_{\substack{i\neq j\\i,j\neq k}}^{n}b_{ik}b_{jk}\big)h_k^2
=\frac{1}{2}\sum_{k=1}^n G(k)\,h_k^2.
\]
\end{proof}
If $\sigma_n\equiv 0$ on $D$, we  have the further conclusion. For convenience, we assume $\lambda_n=0$ on $D$ without loss of generality. Under this convention, observe that, for $k=1, \ldots, n-1$, 
\begin{equation}\label{eq:sigma-k}
  \sigma_k=\sigma_k(\lambda_1,\ldots, \lambda_{n-1},0)=\sigma_{k}(\lambda_1,\ldots,\lambda_{n-1}).  
\end{equation}
The characteristic polynomial associated with $\lambda_1, \ldots,\lambda_{n-1}$ on $M$ is given by
\begin{equation}
\label{eq:P(lambda)-sigma-2}
P_1(x)=\prod_{i=1}^{n-1}(x-\lambda_i)
=x^{n-1}-\sigma_1x^{n-2}+\sigma_2x^{n-3}-\cdots+(-1)^{n-2}\sigma_{n-2}x+(-1)^{n-1}\sigma_{n-1}.
\end{equation}

\begin{lema}\label{x-n-1}
\label{lem-x}Under the assumptions and notation of Proposition \ref{lem-Lr-h-(n-1)}, assume that $\lambda_n = 0$ on \(D\).   Then 
\begin{enumerate}[(i)]
\item
$\Gamma_{ii}^k
=b_{ik} (\sigma_{n-1})_k$, for \(i\neq k\),  \ $i, k\in \{1, \ldots, n\}$;
\item $\displaystyle\langle X,\nabla \sigma_{n-1}\rangle
=\sum_{k=1}^n u_k\,(\sigma_{n-1})_k^2$.

\end{enumerate}
\noindent Here  \[
b_{ik} =
\begin{cases}
\dfrac{(-1)^{n}\,}{\bigl(\lambda_i - \lambda_k\bigr)\, P_1'(\lambda_i)}, & i \neq k, n \\[10pt]
0, & i = n \text{ and } k\neq n.
\end{cases} \text{ and } 
u_k=\begin{cases}(-1)^n\dfrac{P_1''(\lambda_k)}{2P_1'(\lambda_k)^2}, & k\neq n\\
(-1)^{n+1}\dfrac1{P'(0)}, & k=n.
\end{cases}
\]

\end{lema}
\begin{proof} Observe that \[P_1'(x)=\frac{P'(x)}{x}-\frac{P(x)}{x^2}.\]
By Proposition \ref{lem-Lr-h-(n-1)}, \(\Gamma_{ii}^{\,k}\) and \(b_{ik}\) follow from  \(P_1'(\lambda_i)=\dfrac{P'(\lambda_i)}{\lambda_i}\) for  \(i\ne n\)  and \(\lambda_n=0\).

For $k\neq n$, by Propositions \ref{lem-Lr-h-(n-1)} and \ref{lem:partial-fraction-identity},
\[u_k=\sum_{\substack{i=1\\i\neq k}}^n b_{ik}=\sum_{\substack{i=1\\i\neq k}}^{n-1} b_{ik}=\sum_{\substack{i=1\\i\neq k}}^{n-1} \dfrac{(-1)^{n}\,}{\bigl(\lambda_i - \lambda_k\bigr)\, P_1'(\lambda_i)}=(-1)^n\frac{P_1''(\lambda_k)}{2P_1'(\lambda_k)^2}.
\]

\begin{claim}\label{claim-6p}
   \(
\sum_{i=1}^n\tfrac{1}{P'(\lambda_i)} = 0.
\)
\end{claim}

In fact, by  the partial fraction decomposition of $P(x)^{-1}$: 
\[
\frac{1}{P(x)}=
\sum_{i=1}^n \frac{1}{P'(\lambda_i)}\cdot\frac{1}{x-\lambda_i},
\]
we have
\[
1 =\sum_{i=1}^n \frac{1}{P'(\lambda_i)}\cdot\frac{P(x)}{x-\lambda_i}.
\]
Therefore, comparing the coefficients of \(x^{\,n-1}\) on both sides yields \(\sum_{i=1}^{\,n} \tfrac{1}{P'(\lambda_i)}= 0\), as claimed.

On the other hand, by Propositions \ref{lem-Lr-h-(n-1)} and Claim \ref{claim-6p}, 
\[
u_n=\sum_{i=1}^{n-1} b_{in}=\sum_{i=1}^{n-1}\frac{(-1)^n\lambda_i}{\lambda_iP'(\lambda_i)}=(-1)^n\sum_{i=1}^{n-1}\frac1{P'(\lambda_i)}=(-1)^{n+1}\frac1{P'(0)}.
\]
\end{proof}

\begin{lema}\label{cor:(n-1)}
Let $n>3$. Under the assumptions and notation of Proposition \ref{lem-Lr-h-(n-1)}, suppose that $\lambda_n = 0$ on \(D\).   If the eigenvalues \(\lambda_1,\dots,\lambda_{n}\) are all nonnegative or all nonpositive, then on $D$, 
\[G(k)< 0, \qquad k=1,\ldots,n,\]
and
\[\displaystyle\Psi
=\frac{1}{2}\sum_{k=1}^n G(k)\,(\sigma_{n-1})_k^2\le 0.\]
\end{lema}

\begin{proof} From Lemma
\ref{lem-x},  for \(i \neq k\),  \(i, k= 1,\hdots,n\),
\[
b_{ik}=\dfrac{(-1)^n}{\displaystyle (\lambda_i-\lambda_k) P_1'(\lambda_i)}=\dfrac{(-1)^n}{\displaystyle (\lambda_i-\lambda_k) \prod_{\substack{l=1 \\ l \neq i}}^{n-1} (\lambda_i-\lambda_l)}.
\]
Since \(\lambda_1,\dots,\lambda_{n-1}\) are distinct real numbers, it follows from Lemma \ref{lem-TWY} in the appendix,
\begin{equation}\label{L-(n-1)}
  \sum_{\substack{i \neq j \\ i,j \neq k}}^{n-1} b_{ik} b_{jk} < 0.  
\end{equation}
Hence, for $k=1, \ldots, n-1$, using (\ref{L-(n-1)}) and \(b_{nk}=0\), we have
\[
\begin{split}
G(k)
&=\sum_{\substack{i \neq j \\ i,j \neq k}}^n b_{ik} b_{jk}  = \sum_{\substack{i \neq j \\ i,j \neq k}}^{n-1} b_{ik} b_{jk} < 0.
\end{split}
\]
For $k=n$,
\[
G(n) = \sum_{\substack{i \neq j \\ i,j \neq n}}^n b_{in} b_{jn} = \sum_{i \neq j}^{n-1} \dfrac{1}{\lambda_i P_1'(\lambda_i)\lambda_j P_1'(\lambda_j)} = \sum_{i \neq j}^{n-1} \dfrac{1}{P'(\lambda_i) P'(\lambda_j)}.
\]
Set \(d_i=\dfrac{1}{P'(\lambda_i)}, \quad i=1,\dots,n\).
By Claim \ref{claim-6p}, we rewrite $G(n)$ as
\begin{equation*}
G(n) = \sum_{i \neq j}^{n-1} d_id_j = \big( \sum_{i=1}^{n-1} d_i\big)^2 - \sum_{i=1}^{n-1} d_i^2=d_n^2- \sum_{i=1}^{n-1} d_i^2.
\end{equation*}

\smallskip
Since the eigenvalues \(\lambda_1,\dots,\lambda_{n}\) are distinct and all nonnegative or all nonpositive, we have that, in any case, \(|\lambda_1-\lambda_i|<|\lambda_i|, \quad i=2,\dots,n-1\). Then
\[
\begin{split}
|d_1|^{-1}
&=|P'(\lambda_1)|=|\lambda_1 (\lambda_1-\lambda_2) \dots (\lambda_1-\lambda_{n-1})|  \\
&<|\lambda_1 \lambda_2 \dots \lambda_{n-1}|= |P'(0)| = |d_n|^{-1}.
\end{split}
\]
This implies that
\[
d_n^2 < d_1^2 < \sum_{i=1}^{n-1} d_i^2.
\]
Thus 
\(
G(n) = d_n^2 - \sum_{i=1}^{n-1} d_i^2<0.
\)
Therefore
\(
\Psi\le 0 \quad\text{on } D.
\)
\end{proof}

The proof of the following Theorem \ref{teo:M^n2} is analogous to that of Theorem \ref{teo:M^n1-1} with the necessary modifications.
\begin{thm}
\label{teo:M^n2}
Let $M^n$ $(n>3)$ be a closed $n$-dimensional Riemannian manifold. 
Suppose that $\mathfrak{a}$ is a smooth Codazzi symmetric $(0,2)$-tensor field on $M^n$ and $A$ is its associated $(1,1)$-tensor field. 
Assume that the elementary symmetric functions $\sigma_k=\sigma_k(A)$ of $A$, $k=1,\dots,n-2$, are constant, \(\sigma_n=\sigma_n(A) \equiv 0\), and $\sigma_{n-1}=\sigma_{n-1}(A)$ is  not constant on $M$ and not zero for all $p\in M$. Then:
\begin{enumerate}[(i)]
    \item The image of $(-1)^{n-1}\sigma_{n-1}$ is a compact interval $[a,b]$ with $a<b$.
    \item Let $D_0=\{p\in M\mid (-1)^{n-1}\sigma_{n-1}(p)\in (a,b)\} $.  For each $p\in D_0$, in addition to a  zero eigenvalue, denoted by  $\lambda_n=0$,  $A$ has $n-1$ distinct nonzero eigenvalues   \[\lambda_1(p)<\lambda_2(p)<\cdots<\lambda_{n-1}(p). \] 
    Moreover, the eigenvalues $\lambda_i $  are smooth on $D_0$ and locally there exist smooth orthonormal eigenvectors  $e_i$ associated with $\lambda_i$, $i=1, \ldots, n$.  
   \item Let $e_i, i=1,\ldots, n$, be as in (ii). Then the vector field  \[X=\sum_{i=1}^n\nabla_{e_i}e_{i}\] is well-defined on $D_0$.    
    \item There exist two sequences $a_k \text{ and }  b_k$  such that $[a_k, b_k]\subset (a,b)$,  $\displaystyle\lim_{k\to +\infty}a_k=a$ and $\displaystyle\lim_{k\to +\infty}b_k= b$, and  
\[
\limsup_{k\to +\infty} \int_{D_k} \mathrm{div}(X)\le 0, \]
\end{enumerate}
where $D_k:=\{p\in M; (-1)^{n-1}\sigma_{n-1}(p)\in (a_k,b_k) \}$ and  $X$ is defined in (iii).
\end{thm}

\begin{proof}

$(i)$ Since $M^n$ is compact and  $\sigma_{n-1}$ is smooth and not constant,  the image of $(-1)^{n-1}\sigma_{n-1}$  is a compact interval $[a,b]$ with $a<b$.

$(ii)$  For each $p\in M$, since $\sigma_n=0$ and $\sigma_{n-1}\neq 0$, $A$ has a unique zero eigenvalue, denoted by $\lambda_n=0$ without loss of generality. Denote by $\lambda_1(p)\le\cdots\le\lambda_{n-1}(p)$ the other nonzero eigenvalues of $A(p)$.  Then it is known that $\lambda_i$, $i=1,\ldots, n$, are continuous on $M$ and there exists an open dense subset $V\subset M $ on which all $\lambda_i$ are smooth (see Remark \ref{Re-S75}). 
Define the open subset of $M$ by
\[
D:=\big\{p\in M^n \mid \lambda_1(p)<\lambda_2(p)<\cdots<\lambda_{n-1}(p)\big\}.
\]

\begin{claim}\label{claim-6-1}
     $D$ is nonempty.
\end{claim}

Assume, for the sake of contradiction,  that $D=\emptyset$.

We consider the characteristic polynomial associated with $\lambda_1, \ldots,\lambda_{n-1}$ on $M$:
\begin{equation}
\label{eq:P(lambda)-sigma}
P_1(x)=\prod_{i=1}^{n-1}(x-\lambda_i)
=x^{n-1}-\sigma_1x^{n-2}+\sigma_2x^{n-3}-\cdots+(-1)^{n-2}\sigma_{n-2}x+(-1)^{n-1}\sigma_{n-1},
\end{equation}
where $\sigma_1,\dots,\sigma_{n-1}$ are constant on $M$ and $\sigma_{n-1}$ is not constant on $M$.

 Differentiating  the equation $P_1(\lambda_i)=0$ on $V$,  $i=1, \ldots, n-1$,  we have 
\begin{equation}\label{eq:fund-diff-1}
P_1'(\lambda_i)\,d\lambda_i + (-1)^{n-1} d\sigma_{n-1} = 0.
\end{equation}

For each  $p\in V$, since $D=\emptyset$,  there exists some nonzero eigenvalue \(\lambda_i(p)\) of multiplicity at least $ 2$. For this $\lambda_i$, we have
\(
P_1'(\lambda_i)=0.
\)
Hence \(d\sigma_{n-1}(p)=0\).

Thus $d\sigma_{n-1}=0$ on $V$. Since $V$ is an open  dense subset of \(M\),  it follows that \(d\sigma_{n-1}\equiv 0\) on \(M\) and hence \(\sigma_{n-1}\) is constant on \(M^n\), which contradicts the assumption of the theorem. Therefore, Claim \ref{claim-6-1} holds.

 We define the polynomial $Q_1(x)$ of degree $n-1$ by
\[Q_1(x)=x^{n-1}-\sigma_1x^{n-2}+\sigma_2x^{n-3}-\cdots+(-1)^{n-2}\sigma_{n-2}x\]

Then
\[P_1(x)=Q_1(x)+(-1)^{n-1}\sigma_{n-1}.\] 
Consequently,
\[Q_1(0)=0, \  Q_1'(x)=P_1'(x)=\sum_{j=1}^{n-1}\prod_{i\neq j}^{n-1}(x-\lambda_i), \ \text{ and }  Q_1''(x)=P_1''(x).\]

For each $p\in M$, since $(-1)^{n-1}\sigma_{n-1}$ is independent of $x$,  the graph of $P_1(x)$ differs from that of $Q_1(x)$ by a vertical translation of amount  $(-1)^{n-1}\sigma_{n-1}$.

Choose a point $p\in D\neq \emptyset$. Since $P_1$ has $n-1$ distinct real roots at $p$, $P_1$ has $n-2$ distinct critical points $\xi_1<\cdots<\xi_{n-2}$, which are  local extrema that alternate between maxima and minima, satisfying 
\begin{equation}\label{ine-xi-1}
 \lambda_1<\xi_1<\lambda_2<\ldots<\lambda_i<\xi_i<\lambda_{i+1}<\ldots<\lambda_{n-2}<\xi_{n-2}<\lambda_{n-1}.   
\end{equation}
Consequently, $Q_1(x)$ has the same $n-2$ distinct critical points $\xi_1<\ldots<\xi_{n-2}$, again with alternating local extrema. Hence   $\displaystyle Q_1'(x)=(n-1)\prod_{i=1}^{n-2}(x-\xi_i)$ and  $Q_1''(\xi_i)\neq 0$ by observing that $\xi_i$ are distinct.

We define two numbers $\underline{\mathcal{M}}$ and $\overline{\mathfrak{m}}$ which are  the smallest local maximum value of $Q_1$ and  the largest local minimum value of $Q_1$ respectively, that is,

\[
\underline{\mathcal{M}} := \min\{Q_1(\xi_i): Q_1''(\xi_i)<0\}, \qquad
\overline{\mathfrak{m}} := \max\{Q_1(\xi_i): Q_1''(\xi_i)>0\}.
\]

Since the values of the $n-1$ real roots of  $P_1$ depend on the value of   $(-1)^{n-1}\sigma_{n-1}$, we may also consider  $P_1(x)$ as a one-parameter family of polynomials of degree $n-1$, parameterized by $t$:
\begin{equation}\label{defi:P1}
    P_1(x)(t)=Q_1(x)+t,\  t\in [a,b].
\end{equation}

By an argument similar to that in Theorem \ref{teo:M^n1-1}, the following conclusions hold for $P_1$:

(a) 
 \([a,b]\subset[-\underline{\mathcal{M}},-\overline{\mathfrak{m}}].\)

(b) For each $t\in [a,b]$, every root $\lambda_i$ of $P_1$ has multiplicity  at most $2$.  

$P_1$ has $n-1$ distinct real roots if and only if
\[
t\neq-\underline{\mathcal{M}} \ \text{ and } \ t\neq -\overline{\mathfrak{m}}.
\]

It follows that  for $t\in (a,b)\subset  (-\underline{\mathcal{M}},-\overline{\mathfrak{m}})$, $P_1$ has $n-1$ distinct real roots $\lambda_i$, which satisfy \[\lambda_1<\xi_1<\lambda_2<\ldots<\lambda_i<\xi_i<\lambda_{i+1}<\ldots<\lambda_{n-2}<\xi_{n-2}<\lambda_{n-1},\]
where $\xi_i$, $i=1,\ldots,n-2$ are defined in (\ref{ine-xi-1}).

Then \[D_0=\{p\in M\mid (-1)^n\sigma_{n-1}(p)\in (a,b)\}\subset D.\]

(c) If the case $ a=-\underline{\mathcal{M}}$  (or $b=-\overline{\mathfrak{m}}$, respectively) occurs, then at $t=a$ (or $t=b$, respectively)  there exists at least one index $i$ such that $\lambda_i=\lambda_{i+1}=\xi_i$, where $\xi_i$ satisfies
\[
P_1(\xi_i)=0,   P_1'(\xi_i)=0, \text{ and }  P_1''(\xi_i)<0 \ (\text{or } P_1''(\xi_i)>0, \text{ respectively}).
\]

By (b),  $A$ has $n$ distinct eigenvalues on the open subset $D_0$. This implies the smoothness of the eigenvalues $\lambda_i$ and  the local smoothness of the corresponding orthonormal eigenvectors $e_i, i=1, \ldots, n,$ on  $D_0$.

$(iii)$ By the same argument as in the proof of Theorem \ref{T-W-Y}, $X$ is well-defined on $D_0$.

$(iv)$ Noting $(-1)^{n-1}\sigma_{n-1}$ is smooth on $M$, by Sard's lemma, we can choose two sequences of regular values $a_k$ and $b_k$ of $(-1)^{n-1}\sigma_{n-1}$ such that $[a_k, b_k]\subset(a,b)$, $a_k\to a$ and $b_k\to b$.  

The Stokes theorem implies that
\begin{equation}\label{stokes}
    \int_{D_{k}} \mathrm{div}(X)=\int_{\partial D_k}\langle X,\nu\rangle 
\end{equation}
where $D_k=\{p\in M; (-1)^{n-1}\sigma_{n-1}(p)\in(a_k, b_k)\} $ and $\nu$ is the unit outward normal vector of $\partial D_k$.  

Let $\partial D_k^a=\{p\in M; (-1)^{n-1}\sigma_{n-1}(p)=a_k\}$ and  $\partial D_k^b=\{p\in M; (-1)^{n-1}\sigma_{n-1}(p)=b_k\}.$ Then
\[\partial D_k=\partial D_k^a\cup\partial D_k^b.\]

On $\partial D_k^b$,  $\nu=\frac{(-1)^{n-1}\nabla \sigma_{n-1}}{|\nabla \sigma_{n-1}|}$. By Lemma  \ref{lem-x}, 
we have, 
\begin{equation}\label{eq:xvb}
  \begin{split}
\langle X,\nu\rangle
&=\big\langle X,\frac{(-1)^{n-1}\nabla \sigma_{n-1}}{|\nabla \sigma_{n-1}|}\big\rangle =\frac{1}{|\nabla \sigma_{n-1}|}\big[-\frac12\sum_{i=1}^{n-1}\frac{P_1''(\lambda_i)}{P_1'(\lambda_i)^2}(\sigma_{n-1})_i^2+\frac1{P'(0)}(\sigma_{n-1})_n^2\big].
\end{split}  
\end{equation}
Similarly, on $\partial D_k^a$,    $\nu=-\frac{(-1)^{n-1}\nabla \sigma_{n-1}}{|\nabla \sigma_{n-1}|}$. By Lemma  \ref{lem-x}, we have
\begin{equation}\label{eq:xva}
  \begin{split}
\langle X,\nu\rangle
&=\frac{1}{|\nabla \sigma_{n-1}|}\big[\frac12\sum_{i=1}^{n-1}\frac{P_1''(\lambda_i)}{P_1'(\lambda_i)^2}(\sigma_{n-1})_i^2-\frac1{P'(0)}(\sigma_{n-1})_n^2\big].
\end{split}  
\end{equation}

Next we prove the  two claims for the one-parameter family of polynomials $P_1(x)$ with parameter $t$, given by \eqref{defi:P1}.

\begin{claim}\label{claim-p1}
 $\tfrac{P_1''(\lambda_i)}{P_1'(\lambda_i)^2}$, $i=1,\ldots,n-1$,  are bounded from above on $(a, a+\delta]$  for some $\delta>0$.
\end{claim}
We prove the claim for two possible cases of $a$.

Case 1. $a>-\underline{\mathcal{M}}$. In this case,  $\lambda_i(a)$, $i=1,\ldots,n-1$,  are distinct. Hence, as   $t\searrow a$,  $\lambda_i\rightarrow \lambda_i(a)$, and 
\[
\frac{P_1''(\lambda_i)}{P_1'(\lambda_i)^2}\to  \frac{P_1''(\lambda_{i}(a))}{P_1'(\lambda_{i}(a))^2}, \ i=1,\ldots, n-1.
\]
It implies that 
$\tfrac{P_1''(\lambda_i)}{P_1'(\lambda_i)^2}$, $i=1,\ldots,n-1$,  are bounded from above on $(a, a+\delta]$ for some $\delta>0$.

Case 2. $a=-\underline{\mathcal{M}}$. In this case,  as  $t\searrow a$,  $\lambda_i\rightarrow \lambda_i(a), i=1, \ldots, n-1$.
Notice that only for those indices $i$ such that  $Q_1(\xi_i)=\underline{\mathcal{M}}$, we have  $\lambda_i(a)=\lambda_{i+1}(a)=\xi_i$.

When $n$ is even, we have
\begin{equation}\label{ine:odd}
    \lambda_1(a)\leq \lambda_2(a)<\lambda_3(a)\leq \lambda_4(a)<\lambda_5(a)\leq\ldots<\lambda_{n-3}(a)\leq\lambda_{n-2}(a)<\lambda_{n-1}(a),
\end{equation}
where the equality in $ \lambda_i(a)\leq\lambda_{i+1}(a)$  in \eqref{ine:odd} occurs only for those indices $i$ such that $Q_1(\xi_i)=\underline{\mathcal{M}}$. 

For the index  $i$ such that $\lambda_i(a)=\lambda_{i+1}(a)=\xi_i$, we have $Q_1(\xi_i)=\underline{\mathcal{M}}$, $Q_1'(\xi_i)=0$, and $Q_1''(\xi_i)<0$.  Hence as $t\searrow a$,  $\lambda_i\rightarrow \xi_i$ and $\lambda_{i
+1}\rightarrow \xi_i$. Consequently,
\begin{equation}\label{odd-1-1}
 \frac{P_1''(\lambda_i)}{P_1'(\lambda_i)^2}\to -\infty\  \text{ and } \ \frac{P_1''(\lambda_{i+1})}{P_1'(\lambda_{i+1})^2}\to -\infty.   
\end{equation}

For the index $i$ such that $Q_1(\xi_i)> \underline{\mathcal{M}}$, as  $t\searrow a$,  $\lambda_i\rightarrow \lambda_i(a)$ and $\lambda_{i+1}\rightarrow \lambda_{i+1}(a)$. Hence,

\begin{equation}\label{od-2}
   \frac{P_1''(\lambda_i)}{P_1'(\lambda_i)^2}\to  \frac{P_1''(\lambda_{i}(a))}{P_1'(\lambda_{i}(a))^2}\  \text{ and } \ \frac{P_1''(\lambda_{i+1})}{P_1'(\lambda_{i+1})^2}\to \frac{P_1''(\lambda_{i+1}(a))}{P_1'(\lambda_{i+1}(a))^2}. 
\end{equation}

By \eqref{ine:odd} and the alternating nature of the local extrema $\xi_i$, we obtain that for each index $i< n-1$, $\lambda_i=\lambda_i(t)$, with $ t\in (a,b),$ coincides either with the left endpoint  of the interval $[\lambda_i, \lambda_{i+1}]$ such that $Q_1(\xi_i)\geq \underline{M}$, or with the right endpoint  of the interval $[\lambda_{i-1}, \lambda_{i}]$ such that $Q_1(\xi_{i-1})\geq \underline{M}$. Hence  \eqref{odd-1-1} and \eqref{od-2} imply that $\tfrac{P_1''(\lambda_i)}{P_1'(\lambda_i)^2}$, $i< n-1$,  are bounded from above on $(a, a+\delta]$  for some $\delta>0$.

On the other hand, $\lambda_{n-1}$ is the right endpoint of $[\lambda_{n-2},\lambda_{n-1}]$ with $Q_1(\xi_{n-2})\leq \overline{\mathfrak{m}}<\underline{\mathcal{M}}$. Hence as $t\searrow a$, $\lambda_{n-1}\to \lambda_{n-1}(a)$. Thus
\begin{equation}\label{od-3-1}
   \frac{P_1''(\lambda_{n-1})}{P_1'(\lambda_{n-1})^2}\to  \frac{P_1''(\lambda_{n-1}(a))}{P_1'(\lambda_{n-1}(a))^2}. 
\end{equation}
This implies that  $\tfrac{P_1''(\lambda_{n-1})}{P_1'(\lambda_{n-1})^2}$ is bounded  on $(a, a+\delta]$  for some $\delta>0$.

When $n$ is odd, the following observations hold.
\begin{equation}\label{ine:even}
    \lambda_1(a)< \lambda_2(a)\leq\lambda_3(a)< \lambda_4(a)\leq\lambda_5(a)<\ldots<\lambda_{n-3}(a)\leq \lambda_{n-2}(a)<\lambda_{n-1}(a),
\end{equation}
where the equality in $ \lambda_i(a)\leq\lambda_{i+1}(a)$  in \eqref{ine:even} occurs only for those indices $i$ such that $Q_1(\xi_i)=\underline{\mathcal{M}}$.  

For each index $1<i<n-1$, $\lambda_i=\lambda_i(t)$, with $t\in (a,b)$, is either  the left endpoint  of the interval $[\lambda_i, \lambda_{i+1}]$ such that $Q_1(\xi_i)\geq \underline{M}$, or  the right endpoint  of the interval $[\lambda_{i-1}, \lambda_{i}]$ such that $Q_1(\xi_{i-1})\geq \underline{M}$. 

On the other hand, $\lambda_1$ is the left endpoint of $[\lambda_1,\lambda_2]$ with $Q_1(\xi_1)\leq \overline{\mathfrak{m}}$ and  $\lambda_{n-1}$ is the right endpoint of $[\lambda_{n-2},\lambda_{n-1}]$ with $Q_1(\xi_{n-1})\leq \overline{\mathfrak{m}}$, respectively. 

Using the above observations, by an argument similar to the case where $n$ is even, we can prove that $\tfrac{P_1''(\lambda_i)}{P_1'(\lambda_i)^2}$, $i=1,\ldots,n-1$,  are bounded from above on $(a, a+\delta]$  for some $\delta>0$.

We complete the proof of Claim \ref{claim-p1}.

\begin{claim}\label{claim-p2}
 $\tfrac{P_1''(\lambda_i)}{P_1'(\lambda_i)^2}$, $i=1,\ldots,n-1$,  are bounded from below on $[b-\delta,b)$ for some $\delta>0$.    
\end{claim} 

Since the proof of Claim \ref{claim-p2} is analogous to that of Claim \ref{claim-p1}, we omit it and only point out the following observation needed for the argument:

In the case  $b=-\overline{\mathfrak{m}}$, for the index $i$ such that $Q_1(\xi_i)=\overline{\mathfrak{m}} $,  it holds that $Q_1'(\xi_i)=0$ and $Q_1''(\xi_i)>0$.  Hence as $t\nearrow b$, $\lambda_i\to \xi_i$ and  $\lambda_{i+1}\to \xi_i$. It follows that
\begin{equation*}
  \frac{P_1''(\lambda_i)}{P_1'(\lambda_i)^2}\to +\infty\  \text{ and }\  \frac{P_1''(\lambda_{i+1})}{P_1'(\lambda_{i+1})^2}\to +\infty. 
\end{equation*}

Finally, we prove the conclusion of the item.

By \eqref{eq:xva}, Claim \ref{claim-p1}, and Lemma \ref{lema:boundary}, for $k$ sufficiently large,
\begin{equation}\label{ine:sup1}
    \int_{\partial D^a_k}\langle X,\nu\rangle \le C\int_{\partial D^a_k}|\nabla \sigma_{n-1}|\to 0,  \quad \text{as } k \to \infty.
\end{equation}
Similarly, By\eqref{eq:xvb}, Claim \ref{claim-p2}, and Lemma \ref{lema:boundary}, for $k$ sufficiently large,
\begin{equation}\label{ine:sup2}
    \int_{\partial D^b_k}\langle X,\nu\rangle \le C\int_{\partial D^b_k}|\nabla \sigma_{n-1}|\to 0, \quad \text{as }  k\rightarrow \infty.
\end{equation}
Then \eqref{stokes}, together with \eqref{ine:sup1} and \eqref{ine:sup2}, implies that
\[
\limsup_{k\rightarrow \infty}\int_{D_k} \mathrm{div}(X)\leq\limsup_{k\rightarrow \infty}\int_{\partial D^b_k}\langle X,\nu\rangle +\limsup_{k\rightarrow \infty}\int_{\partial D^a_k}\langle X,\nu\rangle= 0.
\]
This completes the proof of Theorem \ref{teo:M^n2}.
\end{proof}

Using the same approach as in the proof of Theorem \ref{teo:M^n}, together with Theorem \ref{teo:M^n2} and Lemma \ref{cor:(n-1)},  we prove Theorem \ref{teo:M^n-(n-1)-intro2} as follows.
\begin{proof}[Proof of Theorem \ref{teo:M^n-(n-1)-intro2}]
We assume for the sake of contradiction that $\sigma_{n-1}(A)$  is not constant. By Lemma \ref{cor:(n-1)}, $\Psi\leq 0$ on $D_0$. By $S_M\ge 0$  and  (\ref{div(X)}):
$\dv (X) =\frac12 S_M - \Psi$, it follows that $\mathrm{div}(X)\ge 0$ on $D_0$. From (iv) of Theorem \ref{teo:M^n2}, we have  $\lim_{k\to +\infty}\int_{D_k}\dv(X)=0$ and then $\int_{D_k} \mathrm{div}(X)=0$ for all $k$. Hence $\dv(X)\equiv 0$ and then $\Psi\equiv 0$ on $D_0$.  By  Lemma \ref{cor:(n-1)}, $(\sigma_{n-1})_i\equiv 0$, $i=1, \ldots, n,$ on $D_0$. Consequently,   $\sigma_{n-1}$ is constant on $D_0$ and by continuity it is constant on $M$, which yields a contradiction.

Therefore, $\sigma_{n-1}(A)$ is constant on $M$, and hence all eigenvalues of $A$ are constant on $M$.

If  $A$ has $n$ distinct eigenvalues somewhere on $M$, then all eigenvalues of $A$ are distinct constants on $M$. By Lemma  \ref{cor:(n-1)}, $\Psi\equiv 0$ on $M$. Hence 
\[\int_M\frac12S_M=\int_M\dv(X)=0.
\]
Since $S_M\geq 0$, we have $S_M\equiv 0$ on $M$. This completes the proof of Theorem~\ref{teo:M^n}.
\end{proof}

In the following, we prove Theorem \ref{teo:M^n-(n-1)-m}.

\begin{proof}[Proof of Theorem  \ref{teo:M^n-(n-1)-m}]
    Without loss of generality, we assume that the maximal eigenvalue of $A$   is a constant simple eigenvalue, denoted by $\lambda_n=c$. Define the smooth Codazzi symmetric $(0,2)$-tensor field on $M$:
    \[b=a-cg,\] where $g$ is the Riemannian metric on $M$.

    Let $B$ be the associated $(1,1)$-tensor field of $b$ on $M$.

    The eigenvalues of $B$ are $\eta_i=\lambda_i-c$, $i=1,\ldots,n$. It follows that $\eta_i<0$ everywhere, for $i=1,\ldots, n-1$, and $\eta_n=0$. Consequently, $\sigma_n(B)\equiv 0$ and $\sigma_{n-1}(B)\neq 0$ everywhere. 
    
    Since subtracting $cg$ shifts all eigenvalues by the same constant,  the  condition $(ii)$ is equivalent to constancy of $\sigma_k(B)$ for $k=1,\dots,n-2$.
    
    Therefore $B$ satisfies the assumptions of Theorem \ref{teo:M^n-(n-1)-intro2} and hence all eigenvalues of $B$ are constant. This implies that  all eigenvalues of $A$ are constant on $M$. Finally, if the eigenvalues of $A$ are distinct at some point, then the eigenvalues of $B$ are distinct at that point. It follows from Theorem~\ref{teo:M^n-(n-1)-intro2} that  $S_M\equiv0$.
    \end{proof}

\section{Applications to shape operator and Ricci tensor }\label{sec:app}

In this section, we give  applications to some smooth Codazzi symmetric $(0,2)$-tensor fields.
\subsection{Hypersurfaces in the sphere}\

Using Theorem  \ref{teo:M^n-(n-1)-m}, we now prove Corollary \ref{cor:hypersphere}, which provides a rigidity  result for closed  isoparametric hypersurfaces in the $(n+1)$-dimensional unit sphere $\mathbb S^{n+1}$.

\begin{proof}[Proof of Corollary \ref{cor:hypersphere}]
Let $M^n \subset \mathbb{S}^{n+1}$ be a closed hypersurface with shape operator $A$. 
Since  $A$ is a Codazz tensor and $H_r=\frac1{\binom{n}{r}}\sigma_r(A)$,  the hypotheses $(i), (ii)$, and $(iii)$ allow us to apply Theorem \ref{teo:M^n-(n-1)-m}. It follows that all principal curvatures of $M$ are constant, and hence $M$ is isoparametric.

If the principal curvatures are distinct at some point, then the last conclusion in Theorem \ref{teo:M^n-(n-1)-m} yields  $S_M\equiv0$.
\end{proof}
\begin{remark}
It was proved by Tang and Yan \cite{Tang-Yan23} that any isoparametric hypersurface in the  unit sphere $\mathbb S^{n+1}$ has nonnegative scalar curvature.

\end{remark}

\subsection{Ricci tensor}\

A Riemannian manifold $(M^n,g)$ is said to have \emph{harmonic curvature} if the Riemann curvature tensor is divergence–free, that is,
\[
\operatorname{div} R = 0.
\]
This condition holds if and only if the Ricci tensor $\operatorname{Ric}$ satisfies the Codazzi equation,
\[
(\nabla_X \operatorname{Ric})(Y,Z)=(\nabla_Y \operatorname{Ric})(X,Z)
\quad \text{for all } X,Y,Z\in TM.
\]
Thus Theorem~\ref{teo:M^n-(n-1)-intro2} directly implies Corollary \ref{cor:ricci}.

We mention that there exist various examples of compact Riemannian manifolds with harmonic curvature and with $\nabla \operatorname{Ric}\neq 0$ . 
In particular, such a metric always exists on the product manifold $S^1\times N$, where $N$ is any compact Einstein manifold with positive scalar curvature.

\section{Appendix}\label{appendix}
For convenience, we include the lemmas of Almeida–Brito \cite{Almeida-Brito90} and Tang–Wei–Yan \cite{Tang-Wei-Yan20}, as well as their respective proofs.

\begin{lema}[Almeida-Brito\cite{Almeida-Brito90}]
\label{lema:advanced}
Let $M$ be a closed Riemannian manifold.
Suppose $u:M\to\mathbb{R}$ is a smooth function and let $m=\min_M u$.
For $\varepsilon>0$, set
\[
D_\varepsilon = u^{-1}([m,m+\varepsilon]).
\]
Then
\begin{equation}
\label{eq:advanced}
\lim_{\varepsilon\to0}\int_{D_\varepsilon}|\Delta u|\,\mathrm{vol}=0.
\end{equation}
\end{lema}

\begin{proof}
Assume, by contradiction, that there exists $\varepsilon_0>0$ such that
\[
\int_{D_\varepsilon}|\Delta u|\,\mathrm{vol}\ge \varepsilon_0
\quad\text{for all sufficiently small }\varepsilon>0.
\]
Set
\[
E=\{\,p\in M:\ u(p)=m,\ \Delta u(p)\ne0\,\}.
\]
For any $p\in E$ we have $du(p)=0$ and $\Delta u(p)\ne0$, so we can choose a local coordinate system $x=(x_1,\ldots,x_n)$ on a neighborhood $V_p$ of $p$ such that $\partial u/\partial x_1$ has no critical points in $V_p$.
Define
\[
\Sigma_p=\Bigl\{\,q\in V_p:\ \frac{\partial u}{\partial x_1}(q)=0\,\Bigr\}.
\]
Then $\Sigma_p$ is a smooth hypersurface and $E\cap V_p\subset\Sigma_p$. Since $u^{-1}(m)$ is compact, finitely many such $V_p$
cover $u^{-1}(m)$. Let $\Sigma$ denote the corresponding finite union of smooth hypersurfaces, and set
\[
\mathcal{U}=\Bigl\{\,q\in M:\ |\Delta u(q)|<\frac{\varepsilon_0}{4\,\mathrm{vol}(M)}\,\Bigr\}.
\]
Choose an open set $\mathcal{V}\subset M$ such that $\Sigma\subset\mathcal{V}$ and $\mathrm{vol}(\mathcal{V})<\tfrac{\varepsilon_0}{4C}$, where $C=1+\max_M|\Delta u|$.
For sufficiently small $\varepsilon>0$,  we have $D_\varepsilon\subset\mathcal{U}\cup\mathcal{V}$, and hence
\[
\int_{D_\varepsilon}|\Delta u|
\le \int_{\mathcal{U}}|\Delta u|+\int_{\mathcal{V}}|\Delta u|
\le \frac{\varepsilon_0}{4} + C\,\mathrm{vol}(\mathcal{V})
< \frac{\varepsilon_0}{2},
\]
contradicting the assumption.  Therefore, the limit \eqref{eq:advanced} holds.
\end{proof}

Now we prove Lemma 2.1 in \cite{Tang-Wei-Yan20}  with a  slightly different proof.
\begin{lema}[Tang-Wei-Yan \cite{Tang-Wei-Yan20}] \label{lem-TWY}
    Let  $\lambda_1,\hdots,\lambda_n$ $(n\geq3)$  be distinct real numbers. For a fixed $r \in \{1,\hdots,n\}$. Define
    \begin{equation*}
L(r) = \sum_{\substack{i \neq j \\ i,j \neq r}}^n c_{ir} c_{jr},
\end{equation*}
where 
\[
c_{ir}=\frac{1}{\displaystyle (\lambda_r-\lambda_i) \prod_{\substack{k=1 \\ k \neq i}}^n (\lambda_k-\lambda_i)}.
\]
Then $L(r)<0$.
\end{lema}

\begin{proof}
Fix $r$ and denote $b_{i} = \frac{1}{\lambda_i - \lambda_r}$, $i \neq r$. Note that $b_i$'s are $n-1$ distinct non-zero numbers. Let $\Lambda = \prod_{i \neq r}^n b_i$.  For $i \neq r$, we have
\begin{eqnarray*}
c_{ir} &=& \dfrac{1}{\displaystyle (\lambda_r-\lambda_i) \prod_{\substack{k=1 \\ k \neq i}}^n (\lambda_k-\lambda_i)} = \dfrac{b_i^2}{\displaystyle \prod_{\substack{k=1 \\ k \neq i,r}}^n \big(\frac{1}{b_k} - \frac{1}{b_i}\big)} \\
&=& \dfrac{\displaystyle b_i^n \prod_{k \neq i,r}^n b_k}{\displaystyle \prod_{\substack{k=1 \\ k \neq i,r}}^n \left(b_i - b_k\right)} 
=\dfrac{\displaystyle b_i^{n-1} \Lambda}{\displaystyle \prod_{\substack{k=1 \\ k \neq i,r}}^n \left(b_i - b_k\right)} .
\end{eqnarray*}
Observe that $b_i, i\neq r$ are $n-1$ distinct roots of the polynomial $\Lambda x^{n-1}-\sum_{\substack{i=1 \\ i \neq r}}^n c_{ir} {\displaystyle \prod_{\substack{k=1 \\ k \neq i,r}}^n \left(x - b_k\right)}$. Then Vieta theorem implies 
\begin{equation*}
\sum_{\substack{i=1 \\ i \neq r}}^n c_{ir} = \Lambda \sum_{\substack{i=1 \\ i \neq r}}^n b_i.
\end{equation*}
Without loss of  generality, we assume  $\sum_{\substack{i=1 \\ i \neq r}}^n b_i>0$ and $b_{i_0}=\max\{ b_i\mid  i\in \{1,2,\cdots, n\}, i\ne r\}$.
\[\begin{split}
 \bigg(\sum_{\substack{k=1 \\ k \neq r}}^n b_k\bigg)  {\displaystyle \prod_{\substack{k=1 \\ k \neq i_0,r}}^n (b_{i_0} - b_k)}&= \bigg(b_{i_0}+\sum_{\substack{k=1 \\ k \neq i_0,r}}^n b_k\bigg)  {\displaystyle \prod_{\substack{k=1 \\ k \neq i_0,r}}^n (b_{i_0} - b_k)}\\
 &=b_{i_0}^{n-1} \bigg(1+\sum_{\substack{k=1 \\ k \neq i_0,r}}^n \frac{b_k}{b_{i_0}}\bigg)  {\displaystyle \prod_{\substack{k=1 \\ k \neq i_0,r}}^n \bigg(1-\frac{b_k}{b_{i_0}}\bigg)}<b_{i_0}^{n-1}.
\end{split}
\]
Here we used the fact that the function $ f(x_1, x_2, \cdots, x_m)=\ln \left[(1+\sum_{i=1}^mx_i){\displaystyle \prod_{i=1}^m \left(1-x_i\right)}\right]$ is  concave on the domain $\{(x_1, x_2, \cdots, x_m)\in \mathbb{R}^m, \sum_{i=1}^mx_i>-1, x_i\le 1\}$ and $(0,0,\cdots, 0)  $ is the unique maximum point where $f(0,0,\cdots,0)=0$.

\noindent Hence
\[\begin{split}
 \bigg(\sum_{\substack{i=1 \\ i \neq r}}^n c_{ir}\bigg)^2&=\Lambda^2\bigg(\sum_{\substack{k=1 \\ k \neq r}}^n b_k\bigg) ^2 \\
 &<\Biggl(\frac{b_{i_0}^{n-1}\Lambda}{\displaystyle \prod_{\substack{k=1 \\ k \neq i_0,r}}^n \left(b_{i_0} - b_k\right)}\Biggr)^2=c_{i_0r}^2<\sum_{\substack{i=1 \\ i \neq r}}^n c_{ir}^2.
\end{split}
\]
Therefore
\begin{equation*}
L(r) = \sum_{\substack{i \neq j \\ i,j\neq r}}^n c_{ir} c_{jr} = \Big(  \sum_{\substack{i=1 \\ i \neq r}}^n c_{ir}\Big)^2 - \sum_{\substack{i=1 \\ i \neq r}}^n c_{ir}^2<0.
\end{equation*}
\end{proof}


\begin{thebibliography}{XXX0000}

\bibitem[AB90]{Almeida-Brito90}
S. C. de Almeida and F. G. B. Brito. "Closed 3-dimensional hypersurfaces with constant mean curvature and constant scalar curvature". In: \textit{Duke Mathematical Journal} 61 (1990), pp. 195-206.

\bibitem[ABS07]{Almeida-Brito-Sousa07}
S. C. de Almeida, F. G. B. Brito and L.A.M. de Sousa Jr. "Closed hypersurfaces of $\mathbb{S}^4$ with two constant curvature functions". In: \textit{Results in Mathematics} 50 (2007), pp. 17-26.

\bibitem[CdK70]{Chern-doCarmo-Koba70}
S. S. Chern, M. do Carmo and S. Kobayashi. "Minimal submanifolds of the sphere with second fundamental form of constant length". In: \textit{Functional Analysis and Related Fields, Springer-Verlag, Berlin} (1970), pp. 59-75.

\bibitem[Cha93a]{Chang93a}
S. P. Chang. "A closed hypersurface with constant scalar curvature and constant mean curvature in $\mathbb{S}^4$ is isoparametric". In: \textit{Communications in Analysis and Geometry} 1 (1993), pp. 71-100.

\bibitem[Cha93b]{Chang93b}
S. P. Chang. "On minimal hypersurfaces with constant scalar curvatures in $\mathbb{S}^4$". In: \textit{Journal of Differential Geometry} 37 (1993), pp. 523-534.

\bibitem[CW93]{Cheng-Wan}
Q.~M.~Cheng and Q.~R.~Wan.
 \emph{Hypersurfaces of space forms $M^{4}(c)$ with constant mean curvature}.
 In \emph{Geometry and Global Analysis}, Tohoku University, Sendai, 1993, pp.~437--442.

\bibitem[CZ23]{Cheng-Zhou23}
X. Cheng and D. Zhou, 
``Rigidity of four-dimensional gradient shrinking Ricci solitons'',
\textit{Journal für die reine und angewandte Mathematik (Crelle’s Journal)} 802 (2023), 255--274.


\bibitem[DGW17]{DGW}
Q. T. Deng, H. L. Gu and Q. Y. Wei. "Closed Willmore minimal hypersurfaces with constant scalar curvature in $\mathbb{S}^5(1)$ are isoparametric". In: \textit{Advances in Mathematics} 314 (2017), pp. 278-305.

\bibitem[DS83]{DS83}A. Derdziński  and  C.-L. Shen,  Codazzi Tensor Fields, Curvature and Pontryagin Forms. Proceedings of the London Mathematical Society, s3-47(1983): 15-26. 

\bibitem[LSS05]{Lusala-Scherfner-Sousa05}
T. Lusala, M. Scherfner, L. A. M. de Sousa Jr. "Closed minimal Willmore hypersurfaces of $\mathbb{S}^5(1)$ with constant scalar curvature". In: \textit{Asian Journal of Mathematics} 9.1 (2005), pp. 65-78.

\bibitem[FG16]{FG16}
M.~Fern\'andez-L\'opez and E.~Garc\'{\i}a-R\'{\i}o,
\emph{On gradient Ricci solitons with constant scalar curvature},
Proc. Amer. Math. Soc. \textbf{144} (2016), no.~1, 369--378,
DOI 10.1090/proc/12693. MR3415603

\bibitem[PW09]{PW09} Peter Petersen and William Wylie, Rigidity of gradient Ricci solitons, Pacific J. Math.
241 (2009), no. 2, 329–345, DOI 10.2140/pjm.2009.241.329. MR2507581

\bibitem[Sim68]{Simons68}
J. Simons. "Minimal varieties in Riemannian manifolds". In: \textit{Annals of Mathematics} 88 (1968), pp. 62-105.

\bibitem[SVW12]{SVW}
M. Scherfner, L. Vrancken and S. Weiss. "On closed minimal hypersurfaces with constant scalar curvature in $\mathbb{S}^7$". In: \textit{Geometriae Dedicata} 161 (2012), pp. 409-416.
\bibitem[SWY12]{SWY}
M. Scherfner, S. Weiss, and S.-T. Yau,
A review of the Chern conjecture for isoparametric hypersurfaces in spheres,
in \emph{Advances in Geometric Analysis},
Adv. Lect. Math. (ALM), vol.~21,
Int. Press, Somerville, MA, 2012, pp.~175--187.

\bibitem[S75]{Singley75}
D.~H.~Singley,
\newblock Smoothness theorems for the principal curvatures and principal vectors of a hypersurface,
\newblock \emph{Rocky Mountain J. Math.} \textbf{5} (1975), no.~1, 135--144.


\bibitem[TWY20]{Tang-Wei-Yan20}
Z. Tang, D. Wei and W. Yan. "A sufficient condition for a hypersurface to be isoparametric". In: \textit{Tohoku Mathematical Journal} 72 (2020), pp. 493-505.

\bibitem[TY23]{Tang-Yan23}
Z. Tang and W. Yan. "On the Chern conjecture for isoparametric hypersurfaces". In: \textit{Science China Mathematics} 66 (2023), pp. 143–162.


\end{thebibliography}
\end{document}